\documentclass[5pt]{article}

\usepackage{geometry} 
\usepackage{graphicx,url}
\usepackage{amsmath,amsthm,amssymb,vmargin}
\usepackage{epstopdf}
\usepackage{cjhebrew}
\usepackage{footnote}

\newtheorem{theorem}{Theorem}
\newtheorem{proposition}{Proposition}[section]
\newtheorem{remark}{Remark}[section]

\newtheorem{condition}{Condition}
\newtheorem{lemma}{Lemma}[section]

\newcommand{\ind}{1\hspace{-.27em}\mbox{\rm l}}

\allowdisplaybreaks

\begin{document}

\title{On the Limiting Behaviour of Needlets Polyspectra\footnote{Research Supported by ERC Grant 277742 \emph{Pascal}}}
\author{Valentina Cammarota and Domenico Marinucci \footnote{Corresponding author. e-mail: marinucc@mat.uniroma2.it}\\
Department of Mathematics, University of Rome Tor Vergata}
\numberwithin{equation}{section}

\maketitle

\abstract{This paper provides quantitative Central Limit Theorems
for nonlinear transforms of spherical random fields, in the high
frequency limit. The sequences of fields that we consider are
represented as smoothed averages of spherical Gaussian
eigenfunctions and can be viewed as random coefficients from
continuous wavelets/needlets; as such, they are of immediate
interest for spherical data analysis. In particular, we focus on
so-called needlets polyspectra, which are popular tools for
nonGaussianity analysis in the astrophysical community, and on the
area of excursion sets. Our results are based on Stein-Maliavin
approximations for nonlinear transforms of Gaussian fields, and on
an explicit derivation on the high-frequency limit of the fields'
variances, which may have some independent interest}.

\begin{itemize}

\item{Keywords and Phrases: Spherical Random Fields,
Stein-Malliavin Approximations, Polyspectra, Excursion Sets,
Wavelets, Needlets}

\item{AMS Classification: : 60G60; 62M15, 42C15}

\end{itemize}

\section{Introduction}

\subsection{Background and Notation}

Let $\left\{ f(x),%
\text{ }x\in S^{2}\right\} $ denote a Gaussian, zero-mean
isotropic spherical random field, i.e. for some probability space
$(\Omega ,\Im ,P)$ the application $f(x,\omega )\rightarrow
\mathbb{R}$ is $\left\{ \Im \times \mathcal{B}(S^{2})\right\} $
measurable, $\mathcal{B}(S^{2})$ denoting the Borel $\sigma
$-algebra on the sphere. We shall use $d \sigma(x)$ to denote the
Lebesgue measure on the sphere which, in spherical coordinates, is
defined as $d \sigma(x):=\sin \theta d \theta d \varphi$. It is
well-known that the following
representation holds, in the mean square sense (see for instance \cite
{leonenko2}, \cite{marpecbook}, \cite{mal}, \cite{mal2}):

$$f(x)=\sum_{{\ell }m}a_{\ell m} Y_{\ell m}(x)=\sum_{\ell=1}^\infty f_{\ell}(x), \hspace{1cm} f_{\ell}(x)=\sum_{m=-\ell}^{\ell} a_{\ell m} Y_{\ell m}(x),$$
where $\left\{ Y_{\mathbb{\ell }m}(.)\right\} $ denotes the family
of spherical harmonics, and $\left\{ a_{\mathbb{\ell }m}\right\} $
the array of
random spherical harmonic coefficients, which satisfy $\mathbb{E}a_{\mathbb{%
\ell }m}\overline{a}_{\mathbb{\ell }^{\prime }m^{\prime }}=C_{\mathbb{\ell }%
}\delta _{\mathbb{\ell }}^{\mathbb{\ell }^{\prime }}\delta
_{m}^{m^{\prime }}$; here, $\delta _{a}^{b}$ is the Kronecker
delta function, and the sequence $\left\{ C_{\mathbb{\ell
}}\right\} $ represents the angular power spectrum of the field.
As pointed out in \cite{mp2012}, under isotropy the
sequence $C_{\mathbb{\ell }}$ necessarily satisfies $\sum_{\mathbb{\ell }} C_\ell \frac{(2
\mathbb{\ell }+1)}{4 \pi}=\mathbb{E}T^{2}<\infty $ and
the random field $f(x)$ is mean square continuous.

The Fourier components $\left\{ f_{\mathbb{%
\ell }}(x)\right\} $, can be viewed as random eigenfunctions of
the spherical
Laplacian:%
\begin{equation*}
\Delta _{S^{2}}f_{\mathbb{\ell }}=-\mathbb{\ell }(\mathbb{\ell }+1)f_{%
\mathbb{\ell }}\text{ , }\mathbb{\ell }=1,2,...;
\end{equation*}%
the random fields $\{f_{\ell }(x), x \in S^2\}$ are isotropic,
meaning that the probability laws of $f_\ell(\cdot)$ and  $f^g_\ell(\cdot):=f_\ell( g \cdot)$ are the same for any rotation $g \in SO(3)$. Also, $\{f_{\ell }(\cdot)\}$ is centred Gaussian, with
covariance function
$$\mathbb{E}[f_{\ell }(x) f_{\ell }(y)]=C_\ell \frac{2\ell+1}{4 \pi}  P_{\ell }(\cos d (x,y))$$
where $P_\ell$ are the usual Legendre polynomials defined dy
Rodrigues' formula
$$P_{\ell }(t):= \frac{1}{2^{\ell } \ell !} \frac{d^l}{d t^l} (t^2-1)^{\ell },$$
and $d(x,y)$ is the spherical geodesic distance between $x$ and
$y$. Their asymptotic behaviour of $f_{\ell}(x)$ and their
nonlinear transforms has been studied for instance by \cite{BGS},
\cite{Wig1}, \cite{Wig2}, see also \cite{MaWi,MaWi2,MaWi3}.

More often, however, statistical procedures to handle with
spherical data are based upon wavelets-like constructions, rather
than standard Fourier analysis. For instance, the
astrophysical/cosmological literature on these issues is vast, see
for instance \cite{Wiaux}, \cite{Starck} and the references
therein. As well-known, indeed, the double localization properties
of wavelets (in real and harmonic domain) turn out to be of great
practical value when handling real data.

In view of these motivations, we shall focus here on sequence of
spherical random fields which can be viewed as averaged forms of
the spherical eigenfunctions, e.g.
\begin{equation*}
\beta _{j}(x)=\sum_{\ell=2^{j-1}}^{2^{j+1}}b(\frac{\mathbb{\ell }}{B^{j}})f_{\mathbb{%
\ell }}(x)\text{ , }j=1,2,3...
\end{equation*}%
for $b(.)$ a weight function whose properties we shall discuss
immediately. The fields $\left\{ \beta _{j}(x)\right\} $ can
indeed be viewed as a
representation of the coefficients from a continuous wavelet transform from $%
T(x),$ at scale $j$, see also the discussion in \cite{marinuccivadlamani}. More precisely, consider the kernel%
\begin{align*}
\Psi _{j}(\left\langle x,y\right\rangle ) &:=\sum_{\mathbb{\ell }}b(\frac{%
\mathbb{\ell }}{B^{j}})\frac{2\mathbb{\ell }+1}{4\pi }P_{\mathbb{\ell }%
}(\left\langle x,y\right\rangle ) \\
&=\sum_{\mathbb{\ell }}b(\frac{\mathbb{\ell
}}{B^{j}})\sum_{m=-\mathbb{\ell
}}^{\mathbb{\ell }}Y_{\mathbb{\ell }m}(x)\overline{Y}_{\mathbb{\ell }m}(y)%
\text{.}
\end{align*}%
Assume that $b(.)$ is smooth (e.g. $C^{\infty })$, compactly supported in $%
[B^{-1},B],$ and satisfying the partition of unity property $\sum_{j}b^{2}(%
\frac{\mathbb{\ell }}{B^{j}})=1,$ for all $\ell >B$. Then $\Psi
_{j}(\left\langle x,y\right\rangle )$ can be viewed as a
continuous version of the needlet transform, which was introduced
by Narcowich et al. in \cite{npw1}, and considered from the point
of view of statistics and cosmological data
analysis by many subsequent authors, starting from \cite{bkmpAoS}, \cite%
{mpbb08}, \cite{pbm06}. In this framework, the following
localization property is now well-known: for all $M\in
\mathbb{N}$, there exists a
constant $C_{M}$ such that%
\begin{equation*}
\left\vert \Psi _{j}(\left\langle x,y\right\rangle )\right\vert \leq \frac{%
C_{M}B^{2j}}{\left\{ 1+B^{j}d(x,y)\right\} ^{M}}\text{ ,}
\end{equation*}%
where $d(x,y)=\arccos (\left\langle x,y\right\rangle )$ is the
usual
geodesic distance on the sphere. Heuristically, the field%
\begin{equation*}
\beta _{j}(x)=\int_{S^{2}}\Psi _{j}(\left\langle x,y\right\rangle
)f(y)dy=\sum_{\mathbb{\ell }}b(\frac{\mathbb{\ell }}{B^{j}})f_{\mathbb{\ell }%
}(x)\text{ }
\end{equation*}%
is then only locally determined, i.e., for $B^{j}$ large enough
its value depends only from the behaviour of $f(y)$ in a
neighbourhood of $x.$ This is a very important property, for
instance when dealing with spherical random fields which can only
be partially observed, the canonical example being provided by the
masking effect of the Milky Way on Cosmic Microwave Background
radiation \cite{PlanckNG, PlanckIS}.

It is hence very natural to produce out of $\left\{ \beta
_{j}(x)\right\} $ nonlinear statistics of great practical
relevance. For instance, it is readily seen that%
\begin{equation*}
\mathbb{E}\left\{ \beta _{j}^{2}(x)\right\} =\sum_{\mathbb{\ell }}b(\frac{%
\mathbb{\ell }}{B^{j}})\frac{2\mathbb{\ell }+1}{4\pi }C_{\mathbb{\ell }}%
\text{ ,}
\end{equation*}%
which hence suggests a natural \textquotedblleft
local\textquotedblright\ estimator for a binned form of the
angular power spectrum. More generally, we might focus on statistics of the form%
\begin{equation*}
\nu_{j;q}:=\int_{S^{2}}H_{q}(\beta _{j}(x))dx\text{ ,}
\end{equation*}%
where $H_{q}(.)$ is the Hermite polynomial of $q$-th order, which
can be labelled needlets polyspectra for a straightforward analogy
with the Fourier case. For $q=3$ we obtain for instance the
needlets bispectrum, which was in introduced in \cite{lanm} and
then widely used on CMB data to study nonGaussian behaviour, see
for instance \cite{rudjord1,rudjord2,donzelli} for more discussion
and details; for $q=4$ we obtain the needlets trispectrum, which
is the natural candidate to estimate higher-order nonGaussian
behaviour such as the one introduced by cubic models through the
parameter $g_{NL}$, see \cite{PlanckNG}. As we shall show below,
the analysis of such polyspectra for arbitraty values of $q$
provides moreover natural building blocks for other nonlinear
functionals of the field $\beta_j(x)$ we shall investigate in
particular quantitative Central Limit Theorems for the excursion
sets, as $j \rightarrow \infty$.

Concerning this point, we stress that the limiting behaviour we
consider is in the high frequency sense, e.g. assuming that a
single realization of a spherical random field is observed at
higher and higher resolution as more and more refined experiments
are implemented. This is the
setting adopted in \cite{marpecbook}, see also \cite{anderes},\cite{loh},%
\cite{wang},\cite{steinm} for the related framework of
fixed-domain asymptotics, and \cite{PlanckNG, PlanckIS} for applications to
 cosmological data analysis.   


\subsection{Statement of the main results}
The main technical contribution of this paper is the derivation of
analytical expressions for the asymptotic variance of the needlet
polyspectra $\nu_{j;q}$. In particular, under suitable regularity
conditions on angular power spectra we shall be able to show the
following result (compare \cite{MaWi2}).

\begin{theorem} \label{00}
For $q \geq 2$,   we have
$$\lim_{j \to \infty}2^{2j}{\rm Var}[ \nu_{j;q}]= q! c_q$$
where $$c_2= \frac{8 \pi^2 }{\left( \int_{\frac 1 2}^{2} b^2(x) x^{1- \alpha} d x \right)^2}  \int_{\frac 1 2}^{2}   b^4(x_1) x_1^{1-2 \alpha}  d x_1;$$
\begin{align*}c_3&=  \frac{16 \pi}{  \left(\int_{\frac 1 2}^{2} b^2(x) x^{1- \alpha} d x \right)^3} \int_{\frac 1 2}^2 \cdots \int_{\frac 1 2}^2 \prod_{i=1}^3 b^2(x_i)  x_i^{1-\alpha}  \\
&\hspace{0.4cm} \times    \frac{1}{\sqrt{  x_1  +   x_2 -  x_3 } \sqrt{  x_1  -   x_2 +  x_3 } \sqrt{-  x_1  +   x_2 +  x_3 } \sqrt{x_1  +   x_2 +  x_3 }} \\
&\hspace{0.4cm} \times  \ind_{P_3}(x_1,x_2,x_3) d x_1 d x_2 d x_3;\end{align*}
\begin{align*}c_4&=  \frac{32}{  \left(\int_{\frac 1 2}^{2} b^2(x) x^{1- \alpha} d x \right)^4}
\int_{\frac{1}{2}}^{2} \cdots \int_{\frac{1}{2}}^{2}  \prod_{i=1}^4   b^2(x_i) x_i^{1-\alpha}   \\
& \hspace{0.4cm} \times  \int_{0}^{4} y    \frac{1}{\sqrt{-x_1+x_2+ y } \sqrt{x_1-x_2+ y } \sqrt{x_1+x_2- y } \sqrt{x_1+x_2+ y }} \\
&\hspace{0.4cm} \times \frac{1}{\sqrt{-x_3+x_4+ y } \sqrt{x_3-x_4+ y } \sqrt{x_3+x_4- y } \sqrt{x_3+x_4+ y }}\\
& \hspace{0.4cm} \times \ind_{P_3}(x_1,x_2, y) \ind_{P_3}(y,x_3,x_4)   \;  dy \; d x_1\cdots d x_4;
\end{align*}
and finally for $q \geq 5$ 
$$c_q=\frac{8 \pi^2 }{ \left( \int_{\frac 1 2}^{2} b^2(x) x^{1-\alpha} d x \right)^q} \int_{\frac 1 2}^2 \cdots  \int_{\frac 1 2}^2  \int_0^\infty  \prod_{k=1}^q b^2(x_k) x_k^{1-\alpha}  J_0(x_k \psi) \psi d \psi d x_1 \cdots d x_q;$$
where $P_3$ is the set of all $(x_1, x_2, x_3) \in \mathbb{R}^3$ that satisfy the `triangular' conditions (\ref{poly}).
\end{theorem}
Here, $J_0$ denotes the standard Bessel function of order zero,
defined as usual by
$$ J_0(x)= \sum _ {k=0}^\infty \frac{(-1)^k x^{2k}}{2^{2k}(k!)^2} \text{ .} $$
For each $q \geq 2$, the scaling factor for the needlets
polyspectra is of order $2^{2j}$. This result can be heuristically
explained as follows. Needlets polyspectra can be viewed as linear
combination of random polynomials of degree $q \times j$. On a
compact manifold as the sphere, there exist exact cubature
formulae for such polynomials, so that the integrals defining
$\nu_{j;q}$ can be really expressed as finite averages sums, of
cardinality $2^{2j}$. In view of the uncorrelation inequality (\ref{corrineq}),
we expect the variance of these averages to scale as the inverse
of the number of summands, e.g. exactly $2^{-2j}$. Making this
heuristic rigorous is indeed quite challenging, and requires a
careful analysis on the behaviour of Legendre polynomials (Hilb's
asymptotics, see \cite{MaWi2, MaWi3}) and Clebsch-Gordan/Wigner's coefficients.

Once the asymptotic behaviour of the variance is established, in
view of the celebrated results from Nourdin and Peccati
\cite{nourdinpeccati} the derivation of quantitative Central
Limit Theorems and total variation/Wasserstein distances limits
requires only the analysis of fourth-order cumulants. These
computations are quite standard and provided in Section \ref{TV}, where
it is hence shown that
\begin{theorem}
For $\mathcal{N}$ a standard Gaussian random variable, as $j
\rightarrow \infty$, we have that
\begin{align*}
d_{TV} \left(  \frac{\nu_{j;q}}{\sqrt{Var(\nu_{j;q})}},
\mathcal{N} \right) = O(2^{-2j}),
\end{align*}
\end{theorem}
\noindent $d_{TV}$ denoting as usual Total Variation distance between random
variables, see below for details and definitions. While this
result is quite straightforward given the previous computations on
the asymptotic variance, it has several statistical applications
for handling Gaussian random fields data, where wavelets
polyspectra are widely exploited.

It is also possible to establish a more challenging result on the
behaviour of excursion sets, which we expand in the $L^2$ sense in
terms of the polyspectra. More precisely, let us define the
empirical measure $\Phi_j(z)$ as follows: for all $z \in
(-\infty,\infty)$ we have
$$\Phi_j(z):=\int_{S^2} \ind_{\{ \tilde{\beta}_j(x) \le z\}} d \sigma(x),$$
where $\tilde{\beta}_j(x)$ has been normalized to have unit
variance; the function $\Phi_j(z)$ provides the (random) measure
of the set where $\tilde{\beta}_j$ lies below the value $z$. We
shall hence be able to prove the following

\begin{theorem} \label{WD}
For $\mathcal{N}$ a standard Gaussian random variable, as  $j \to
\infty$ we have
$$d_W \left( \frac{{\tilde \Phi}_j(z)}{\sqrt{\text{Var} [{\tilde \Phi}_j(z)]}}, \mathcal{N}  \right)=O\left(\frac{1}{\sqrt[4]{j}} \right),$$
\end{theorem}
\noindent $d_{W}$ denoting Wasserstein distance between random variables.
This result is close in spirit to some recent work by Viet-Hung Pham
\cite{V-HP}, who considered a Euclidean setting and
traditional large-sample asymptotics; we exploit several ideas
from his proof in our argument below.

\section{Malliavin operators and quantitative Central Limit Theorems}
In a number of recent papers, summarized in the monograph by
Nourdin and Peccati \cite{nourdinpeccati}, a
beautiful connection has been established  between Malliavin calculus and 
the Stein method
to prove quantitative Central Limit Theorems on functional of
Gaussian subordinated random processes.  In this section we review
briefly some notation on isonormal Gaussian processes and
Malliavin operators and we state the main results on Normal
approximations on Wiener chaos, which we shall exploit in the
sequel of the paper; we follow closely \cite{noupebook}.

Let $\mathfrak H$ be a real separable Hilbert space, with inner product $\langle \cdot, \cdot \rangle_{\mathfrak H}$. An {\it isonormal Gaussian process} over $\mathfrak H$ is a collection $X=\{X(h): h \in {\mathfrak H}\}$ of jointly Gaussian random variables defined on some probability space $(\Omega, {\mathcal F}, \mathbb{P})$, such that $\mathbb{E} [X(h) X(g)] = \langle h,g  \rangle_{\mathfrak H}$ for every $h,g \in {\mathfrak H}$. We assume that ${\cal F}$ is generated by $X$.

If $A$ is a Polish space (e.g. complete, metric and separable), $\cal A$ the associated $\sigma$-field and  $\mu$ a positive, $\sigma$-finite and non-atomic measure,  then $\mathfrak{H}=L^2(A,\cal{A},\mu)$ is a real separable Hilbert space with  inner product $\langle g,h\rangle_{\mathfrak{H}}=\int_A g(a) h(a) \mu(da)$. For every $h \in \mathfrak{H}$ it is possible to define the isonormal Gaussian process
\begin{align} \label{W} X(h)=\int_A  h(a) W(da) \end{align}
to be the Wiener-It\^o integral of $h$ with respect to the Gaussian family $W=\{W(B): B \in {\cal A}, \mu(B)< \infty \}$ such that for every $B,C \in {\mathcal A}$ of finite $\mu$-measure $\mathbb{E}[W(B) W(C)]=\mu(B \cap C)$.

Throughout this paper, we shall make extensive use of Hermite polynomials $H_q(x)$. We recall the usual definition: $H_0(x)=1$ and, for every integer $q \ge 1$,
$$H_q(x)=(-1)^q \phi^{-1}(x) \frac{d^q}{dx^q} \phi(x),$$
where $\phi(x)$ is the probability density function of a standard Gaussian variable.

For each $q\ge 0$ the $q$-th {\it Wiener chaos} ${\mathcal{H}_q}$ of $X$ is the closed linear subspace of $L^2(\Omega, {\mathcal F}, \mathbb{P})$ generated by the random variables of type $H_q(X(h))$, $h \in {\mathfrak H}$ such that $||h||_{\mathfrak H}=1$.

The following property of Hermite polynomials is useful for our discussion (for a proof see \cite{noupebook}, Proposition 2.2.1). 
\begin{proposition} \label{j-g}
Let $Z_1,Z_2 \sim {\cal N}(0,1)$ be jointly Gaussian. Then, for all $n,m \ge 0$
\begin{align} \label{EH} \mathbb{E} [H_n(Z_1) H_m(Z_2)]= n! \{\mathbb{E}[Z_1 Z_2]\}^n\end{align}
if $m=n$ and  $\mathbb{E} [H_n(Z_1) H_m(Z_2)]= 0$ if $n \ne m$.
\end{proposition}

The next result is the well known {\it Wiener-It{\^o}}
decomposition  of $L^2(\Omega, {\mathcal F}, \mathbb{P})$ (see
e.g. \cite{noupebook} Theorem 2.2.4 for a proof). Every
random variable $F \in L^2(\Omega, {\mathcal F}, \mathbb{P})$
admits a unique expansion of the type
$$F=\mathbb{E}[F]+\sum_{q=1}^\infty F_q$$ where $F_q \in {\cal
H}_q$ and the series converges in $L^2(\Omega, {\mathcal F},
\mathbb{P})$.

We denote by $\mathfrak{H}^{\otimes q}$ and $\mathfrak{H}^{\odot q}$ the $q$-th tensor product and the $q$-th symmetric tensor product of $\mathfrak{H}$ respectively.  In particular if $\mathfrak{H}=L^2(A,\cal{A},\mu)$ then $\mathfrak{H}^{\otimes q}$ can be identified with $L^2(A^q, {\cal A}^q, \mu^q)$.  For every $1 \le p \le q$, $f \in L^2(A^p, {\cal A}^p, \mu^p)$, $g \in L^2(A^q, {\cal A}^q,\mu^q)$ and $r=1, \dots,p$, the {\it contraction} of the elements $f$ and $g$ is given by
\begin{align*}
&f \otimes_r g(a_1, \dots,a_{p+q-2r})\\
&\;\;=\int_{A^r} f(x_1, \dots,x_r,a_1, \dots, a_{p-r}) g(x_1, \dots,x_r,a_{p-r+1},\dots, a_{p+q-2r}) d \mu(x_1) \dots d \mu(x_r).
\end{align*}
For $p=q=r$ we have $f \otimes_r g=\langle f,g \rangle_{\mathfrak{H}^{\otimes r}}$ and for $r=0$ we have  $f \otimes_0 g=f \otimes g$. Denote with $f \tilde{\otimes}_r g$ the canonical symmetrization of  $f{\otimes}_r g$.

Let $\cal S$ be the set of smooth random variables of the form
$$f(X(h_1), \dots, X(h_m))$$
where $m \ge 1$, $f: \mathbb R^m \to \mathbb R$ is a $C^\infty$ function such that its partial derivatives have at most polynomial growth, and $h_1, \dots, h_m \in \mathfrak{H}$.

Let $L^2(\Omega,{\cal F},\mathbb{P};\mathfrak{H}^{ \odot q})$ be the $\mathfrak{H}^{ \odot q}$-valued random elements $Y$ that are ${\cal F}$-measurable and such that $\mathbb{E}||Y||^2_{\mathfrak{H}^{ \odot q}}<\infty$. For $F \in \mathcal {S}$ and $q \ge 1$, the {\it  $q$-th Malliavin derivative} of $F$ with respect to $X$ is the element of $L^2(\Omega,{\cal F},\mathbb{P};\mathfrak{H}^{ \odot q})$ defined by
$$D^q F=\sum_{i_1, \dots, i_q=1}^m \frac{\partial^q f}{\partial x_{i_1} \dots \partial x_{i_q}} (X(h_1), \dots, X(h_m)) h_{i_1} \otimes \dots \otimes h_{i_q}.$$
If $q=1$, we write $D$ instead of $D^1$.

Let $q \ge 1$ be an integer. We denote by $\text{Dom } \delta^q$ the subset of $L^2(\Omega,{\cal F},\mathbb{P};{\mathfrak H}^{\otimes q})$ composed of those elements $u$ such that there exists a constant $c>0$ satisfying
$$|{\mathbb E}[\langle D^q F,u  \rangle_{\mathfrak{H}^{\otimes q}}]|\le c \sqrt{\mathbb{E}[F^2]},$$
for all $F \in {\cal S}$. If $u \in \text{Dom } \delta^q$, then $\delta^q (u)$ is the unique element of   $L^2(\Omega, \cal{F}, \mathbb{P})$ characterized  by the following  integration by parts formula
$$\mathbb{E}[F \delta^q (u)]=\mathbb{E}[\langle D^q F,u  \rangle_{\mathfrak{H}^{\otimes q} }],$$
for all $F \in \cal{S}$, $\delta^q$ is the {\it divergence operator} of order $q$ . Let $q \ge 1$ and $f \in \mathfrak{H}^{\odot q}$. The $q$-th {\it multiple integral} of $f$ with respect to $X$ is defined by $I_q(f)=\delta^q(f)$. If $f \in L^2(A^q,{\cal A}^q,\mu^q)$ is symmetric, and we regard the Gaussian space generated by the paths of $W$ as an isonormal Gaussian process over $\mathfrak{H}=L^2(A,{\cal A},\mu)$, then
$$I_q(f)=\int_{A^q} f(a_1,\dots, a_q) dW(a_1) \cdots dW(a_q),$$
(see \cite{noupebook}, page 39).

We state now two fundamental properties of multiple integrals that
we shall exploit in the sequel. For a proof see again
\cite{noupebook}, Theorem 2.7.4 and Theorem 2.7.5. Let $q \ge
1$ and $f \in \mathfrak{H}^{\odot q}$, for all $r \ge 1$, we have
\begin{align} \label{DI}
D^r I_q(f)=\frac{q!}{(q-r)!} I_{q-r}(f)
\end{align}
if $r \le q$ and $D^r I_p(f)=0$ if $r>q$. For $1 \le q \le p$, $f \in \mathfrak{H}^{\odot p}$ and $g \in \mathfrak{H}^{\odot q}$ we have
\begin{align} \label{EI}
\mathbb{E}[I_p(f) I_q(g)]=p! \langle f,g\rangle_{\mathfrak{H}^{\otimes p}}
\end{align}
if $p=q$ and $\mathbb{E}[I_p(f) I_q(g)]=0$ if $p \ne q$. The linear operator $I_q$ provides an isometry from $\mathfrak{H}^{\odot q}$ onto $q$-th Wiener chaos ${\cal H}_q$ of $X$. In fact, let $f \in {\mathfrak H}$ be such that $||f||_{\mathfrak{H}}=1$, then for any integer $q \ge 1$, we have
\begin{align} \label{HI}
H_q(X(f))=I_q(f^{\otimes q}),
\end{align}
see once more \cite{noupebook}, Theorem 2.7.7. In particular,
if $f \in L^2(A,{\cal A}, \mu)$,  for $(a_1,\dots,a_q) \in A^q$,
we have
$$f^{\otimes q}(a_1,\dots,a_q)=f(a_1) \cdots f(a_q).$$

The following well-known {\it product formula} implies, in particular, that the product of two multiple integrals is indeed a finite sum of multiple integrals. In fact for $p,q \ge 1$, $f \in \mathfrak{H}^{\odot p}$ and $g \in \mathfrak{H}^{\odot q}$ we have
\begin{align} \label{P}
I_{p}(f) I_q(g)=\sum_{r=0}^{p \wedge q} r! \binom{q}{r} \binom{p}{r} I_{p+q-2 r}(f \tilde{\otimes}_r g ).
\end{align}
For a proof see \cite{noupebook}, Theorem 2.7.10.

We say that $F \in L^2(\Omega, {\cal F}, \mathbb{P})$ belongs to $\text{Dom }L$ if $\sum_{q=1}^\infty q^2 \mathbb{E}[F_q^2] < \infty$, for such an $F$ we define the {\it Ornstein-Uhlenbeck operator} $L F=-\sum_{q=1}^\infty p F_q$. The {\it pseudo-inverse} of $L$ is defined as $L^{-1}F=-\sum_{q=1}^\infty \frac{1}{q} F_q$.\\

Let $\cal{N}$ be a standard Gaussian random variable and define as usual the Kolmogorov, Total Variation and Wasserstein distance, between  $\cal{N}$ and a random variable $F$, as
\begin{align*}
d_W(F,\cal{N})&=\sup_{z \in \mathbb{R}} |\mathbb{P}(F\le z)-\mathbb{P}({\cal N}\le z)|, \\
d_{TV}(F,\cal{N})&=\sup_{B \in {\cal B}(\mathbb{R})} |\mathbb{P}(F \in B)-\mathbb{P}({\cal N} \in B) |,\\
d_{Kol}(F,\cal{N})&=\sup_{h \in Lip(1)} |\mathbb{E}[h(F)]-\mathbb{E}[h({\cal N})]|.
\end{align*}

The connection between stochastic calculus and probability metrics
is summarized in the following proposition (\cite{noupebook},
Theorem 5.1.3.). Let $\mathbb{D}^{1,2}$ be the space of Gaussian
subordinated random variables whose Malliavin derivative has
finite second moment; we have that:

\begin{proposition} \label{N&P}
Let $F \in \mathbb{D}^{1,2}$, such that $\mathbb{E}[F]=0$ and $\mathbb{E}[F^2]=\sigma^2>0$. Then
\begin{align*}
d_W(F,{\cal N}) \le \frac{\sqrt{2}}{\sigma \sqrt{\pi}} \mathbb{E} [|\sigma^2- \langle DF,-D L^{-1}F \rangle_{\mathfrak{H}}|].
\end{align*}
Assuming that $F$ has a density we have
\begin{align*}
d_{TV}(F,{\cal N}) \le \frac{2}{\sigma^2} \mathbb{E} [|\sigma^2- \langle DF,-D L^{-1}F \rangle_{\mathfrak{H}}|], \\
d_{Kol}(F,{\cal N}) \le \frac{1}{\sigma^2} \mathbb{E} [|\sigma^2- \langle DF,-D L^{-1}F \rangle_{\mathfrak{H}}|].
\end{align*}
\end{proposition}

\section{Needlets Random Fields and Wiener Chaoses}

As motivated earlier, in this paper we shall focus on sequences of
needlet random fields, defined by a sequence of spherical random
fields which can be viewed as averaged forms of the spherical
eigenfunctions, e.g. they take the form
$$\beta_j(x)=\sum_{\ell=2^{j-1}}^{2^{j+1}} b(\frac{\ell}{2^j}) f_\ell(x), \hspace{1cm} \tilde{\beta}_j(x):=\frac{\beta_j(x)}{\sqrt{\mathbb{E}[ \beta^2_j(x)] }}, \hspace{1cm}j=1,2,3 \dots$$
for a weight function $b(\cdot)$ such that $b(\cdot)$ is smooth ($b(\cdot) \in C^{\infty}$) compactly supported in $[\frac 1 2,2]$,
and satisfies the partition of unity property $\sum_j
b^2(\frac{\ell}{2^j})=1$, for all $\ell>2$, see also \cite{marinuccivadlamani}. To investigate the correlation, we
introduce some mild regularity conditions on the power spectrum
$C_\ell$  (see \cite{marpecbook}, page 257).

\begin{condition}  \label{con}
There exists $M \in \mathbb{N}$, $\alpha>2$ and a sequence of functions $\{g_j(\cdot)\}$ such that for $2^{j-1}<\ell<2^{j+1}$
$$C_\ell=\ell^{-\alpha} g_j(\frac{\ell}{2^j})>0$$
where $c_0^{-1} \le g_j \le c_0$ for all $j \in \mathbb{N}$  and for some $c_1, \dots, c_M>0$ and  $r=1, \dots, m$, we have $$\sup_j \sup_{2^{-1} \le u \le 2} |\frac{d^r}{du^r} g_j(u)|\le c_r.$$
\end{condition}
Condition \ref{con} entails a weak smoothness requirement on the behaviour of the angular power spectrum, wich is satisfied by cosmologically relevant models. This condition is fulfilled for instance by  models of the form
$$C_\ell=\ell^{-\alpha} G(\ell),$$
where $G(\ell)={P(\ell)}/{Q(\ell)}$ and $P(\ell), Q(\ell)>0$ are two positive polynomials of the same order, let $G:=\lim_{\ell \to \infty} G(\ell)$.

The following property is well-known and gives an upper bound on the correlation coefficient of $\{\beta_j(\cdot)\}$, for a proof see \cite{marpecbook}, Lemma 10.8. 
\begin{proposition} \label{con1}
Under Conditions \ref{con}, for all $M \in \mathbb{N}$, there exists a constant  $C_M>0$ such that the following inequality holds
\begin{align} \label{corrineq}
|\text{Corr }[\beta_j(x) \beta_j(y)]| \le \frac{C_M}{(1+2^j d( x , y  ))^M},
\end{align}
where $d(x,y)=\arccos(\langle x,y \rangle)$ is the geodesic distance on the sphere.
\end{proposition}

Since $\{f_\ell(x)\}$ is Gaussian for each $x \in S^2$, then  $\tilde{\beta}_j(x)$ is a standard Gaussian random
variable and $\beta_j(x)$ is centred 
with variance
\begin{align*}
\mathbb{E}[\beta^2_j(x)]&=\sum_{\ell} b^2(\frac{\ell}{2^j})  C_\ell  \frac{2\ell+1}{4 \pi} P_\ell(\langle x,x \rangle)=\sum_{\ell} b^2(\frac{\ell}{2^j})  C_\ell  \frac{2\ell+1}{4 \pi},
\end{align*}
with $c_1 2^{j(2-\alpha)} \le \mathbb{E}[\beta^2_j(x)] \le c_2 2^{j(2-\alpha)}$. From Proposition \ref{con1}, for the covariance function we have
\begin{align} \label{mu1}
\mathbb{E}[\beta_j(x) \beta_{j}(y)]&=\sum_{\ell} b^2(\frac{\ell}{2^j})  C_\ell  \frac{2\ell+1}{4 \pi} P_\ell(\langle x,y \rangle) \le \frac{C_M}{(1+2^j d(x,y))^M}  \sum_{\ell} b^2(\frac{\ell}{2^j})  C_\ell  \frac{2\ell+1}{4 \pi}.
\end{align}
As in \cite{marinuccivadlamani}, we exploit here the fact that the
field $\{ {\tilde \beta}_j(\cdot)\}$ can be expressed as an
isonormal Gaussian process. Let
$$B_j=\sum_{\ell=2^{j-1}}^{2^{j+1}} b^2(\frac{\ell}{2^j})C_\ell
\frac{2\ell+1}{4 \pi}$$ and for all $x \in S^2$ let us define
\begin{align} \label{Psi}
{\tilde \Theta}_j(\langle x, \cdot \rangle):=\frac{1}{\sqrt{B_j}} \sum_{\ell} b(\frac{\ell}{2^j}) \sqrt{C_\ell} \frac{2 \ell+1}{4 \pi} P_\ell(\langle x, \cdot \rangle)=:\frac{1}{\sqrt{B_j}} \Theta_j(\langle x, \cdot \rangle).
\end{align}
We have that ${\tilde \Theta}_j(\langle x, \cdot \rangle)$ is in the Hilbert space $\mathfrak{H}= L^2(S^2, d\sigma(y))$ and we can represent $\{\tilde{\beta}_j(\cdot)\}$ as
$${\tilde \beta}_j(x)=\int_{S^2} {\tilde \Theta}_j(\langle x,y \rangle) W(d \sigma(y)), \hspace{1cm} x \in S^2,$$
where $W$ is  Gaussian white noise on the sphere as in formula (\ref{W}). In fact the covariance function is given by
\begin{align} \label{Theta}
\tilde{\rho}_j(\langle x,y \rangle)&:=\mathbb{E}[{\tilde \beta}_j(x) {\tilde \beta}_j(y)]=\langle  {\tilde \Theta}_j(\langle x,z \rangle)  {\tilde \Theta}_j(\langle z ,y \rangle) \rangle_\mathfrak{H}=\int_{S^2} {\tilde \Theta}_j(\langle x,z \rangle)  {\tilde \Theta}_j(\langle z ,y \rangle) d \sigma(z) \nonumber \\
&=\frac{1}{{B_j}} \sum_\ell b^2(\frac{\ell}{2^j}) C_\ell \frac{2\ell+1}{4 \pi} P_\ell(\langle x,y \rangle)=:\frac{1}{B_j} \rho_j(\langle x ,y \rangle).
\end{align}
It follows immediately that that the transformed process
$\{H_q({\tilde \beta}_j(\cdot))\}$ belongs to the $q$-th order
Wiener chaos generated by the Gaussian measure governing $f_\ell$ and
so does any linear transform including $$\nu_{j;q}=\int_{S^2}
H_q(\tilde{\beta}_j(x)) d \sigma(x).$$
Let $\ind_{\{ \cdot\}}$ be the usual the indicator function, clearly $\ind_{\{ \tilde{\beta}_j(x) \le z \}}$ belongs for each $x$ and $z \in S^2$ to the $L^2$ space of square integrable functions of Gaussian random variables and we can write
$$\ind_{\{ \tilde{\beta}_j(x) \le z\}}=\sum_{q=0}^\infty \frac{{\cal J}_q(z)}{q!} H_q(\tilde{\beta}_j(x))$$
where the right hand side converges in the $L^2$ sense i.e.
$$\lim_{N \to \infty} \mathbb{E} \left[ \sum_{q=N}^\infty \frac{{\cal J}_q(z)}{q!} H_q({\tilde \beta}_j(x)) \right]^2=0$$ uniformly w.r.t. $x$, $z$. It is possible to provide analytic expressions of the coefficients $\{{\cal J}_q(\cdot)\}$, indeed for $q \ge 1$
$${\cal J}_q(z)=\int_{\mathbb{R}} \ind_{(u \le z)} H_q(u) \phi(u) du=-H_{q-1}(z) \phi(z)$$
and ${\cal J}_0(z)=\Phi(z)$ where $\phi$, $\Phi$ denote, respectively, the density and the cumulative
distribution function of the standard Gaussian (see
\cite{MaWi,MaWi3}). Let us define the empirical measure
$\Phi_j(z)$ as follows: for all $z \in (-\infty,\infty)$ we have
$$\Phi_j(z):=\int_{S^2} \ind_{\{ \tilde{\beta}_j(x) \le z\}} d \sigma(x).$$
The function $\Phi_j(z)$ provides the (random) measure of the set where $\tilde{\beta}_j$ lies below the value $z$. The value $\Phi_j(z)$ at $z=0$ is related to the so-colled defect (or `signed area') of the function $\tilde{\beta}_j: S^2 \to \mathbb{R}$, which is defined by
$${\cal D}_j:=\text{meas} ({\tilde \beta}_j^{-1}(0,\infty))-\text{meas} (\tilde{\beta}_j^{-1}(-\infty,0))$$
and is hence the difference between the areas of positive and negative inverse image of $\tilde{\beta}_j$. By a straightforward transformation we have ${\cal D}_j=4 \pi-2 \Phi_j(0)$. Instead $4 \pi-  \Phi_j(z)$ provides the area of the excursion set $\{x: \tilde{\beta}_j(x)>z\}$. 


\section{On the variance of $\nu_{j;q}$} \label{variance}

In this section we obtain, for all fixed $q \ge 2$, the explicit
value for the limit of $2^{2j} \text{Var }[\nu_{j;q}]$ as $j \to
\infty$.

\begin{theorem} \label{2}
For $q >4$,  we have
$$\lim_{j \to \infty}2^{2j}{\rm Var}[\nu_{j;q}]= q! c_q$$
where $$c_q= \frac{8 \pi^2 }{ \left( \int_{\frac 1 2}^{2} b^2(x) x^{1-\alpha} d x \right)^q} \int_{\frac 1 2}^2 \cdots  \int_{\frac 1 2}^2  \int_0^\infty  \prod_{k=1}^q b^2(x_k) x_k^{1-\alpha}  J_0(x_k \psi) \psi d \psi d x_1 \cdots d x_q.$$
\end{theorem}

\begin{remark}
It is obvious that $c_q \ge 0$ for all $q>0$. In the sequel, se  shall assume that the inequality is strict when needed, e.g., in Theorem \ref{ccc}.  
\end{remark}
\noindent Our proof is close to the argument by \cite{MaWi2}; in particular let us start by recalling the following fact on the asymptotic behaviour of Legendre polynomials (see for instance \cite{szego}, \cite{Wig1}, \cite{Wig2}).

\begin{lemma}[Hilb's asymptotics] \label{hilb}
For any $\varepsilon>0$, $C>0$ we have
$$P_{\ell_k}(\cos \theta)=\left(\frac{\theta}{\sin \theta}  \right)^{\frac 1 2} J_0((\ell_k+1/2) \theta)+\delta_k(\theta)$$
where
$$\delta_k(\theta) \ll \begin{cases}  \theta^2 & 0<\theta<1/\ell_k \\ \theta^{1/2} \ell_k^{-3/2}  & \theta > 1/\ell_k \end{cases}$$
uniformly w.r.t. $\ell_k \ge 1$ and $\theta \in [0,\pi-\varepsilon]$.
\end{lemma}

\begin{lemma} \label{lemma var} Let $q > 4$. For $\ell=2^j$, $\ell_k \in [2^{j-1},2^{j+1}]$ where $k=1,\dots,q$, as $j \to \infty$, we have
\begin{align} \label{PlJ0}
&\ell^2 \int_{0}^{\frac \pi 2}  P_{\ell_1}(\cos \theta) \cdots P_{\ell_q} (\cos \theta)   \sin \theta d \theta=   \int_{0}^{\frac {\ell \pi} 2}  \prod_{k=1}^q J_0\left( \frac{\ell_k+1/2}{\ell}  \psi \right)  \psi d \psi+O(\frac{1}{\sqrt{\ell}}).
\end{align}
\end{lemma} \label{lll}
\begin{proof} From Lemma \ref{hilb} we have
 \begin{align*}
& \int_{0}^{\frac \pi 2}  P_{\ell_1}(\cos \theta) \cdots P_{\ell_q} (\cos \theta)   \sin \theta d \theta \\
&= \int_{0}^{\frac \pi 2}  \prod_{k=1}^q  \left[ \left(\frac{\theta}{\sin \theta}\right)^\frac{1}{2} J_0((\ell_k+1/2) \theta) + \delta_k(\theta) \right]   \sin \theta d \theta\\
 &=\int_{0}^{\frac \pi 2}  \left[ \prod_{k=1}^q  \delta_k(\theta) +  \sum_{k=1}^q \left(\frac{\theta}{\sin \theta}\right)^\frac{1}{2} J_0((\ell_k+1/2) \theta)  \prod_{k' \ne k } \delta_{k'}(\theta) + \cdots + \left(\frac{\theta}{\sin \theta}  \right)^{\frac{q}{2}}  \prod_{k=1}^q  J_0((\ell_k+1/2) \theta)  \right]  \sin \theta d \theta.
\end{align*}
Let, for $k=1, \dots, q$,
\begin{align*}
A_{q-k,k}&:=  \int_{0}^{\frac \pi 2}   \left(\frac{\theta}{\sin \theta}  \right)^{\frac k 2}  \prod_{m=1}^k  J_0((\ell_m+1/2) \theta)   \prod_{m'=k+1}^q \delta_{m'}(\theta) \sin \theta d \theta, \\
A_{q,0}&:=  \int_{0}^{\frac \pi 2}      \prod_{m'=1}^q \delta_{m'}(\theta) \sin \theta d \theta.
\end{align*}

\begin{itemize}
\item For $k=q$,  with the change of variable $\psi=\ell \theta$, we have
\end{itemize}
\begin{align*}
 A_{0,q}&=\frac{1}{\ell} \int_{0}^{\frac {\ell \pi} 2}  \left(\frac{{\psi}/{\ell}}{\sin({\psi}/{\ell})}  \right)^{\frac q 2}   \prod_{m=1}^q  J_0\left( \frac{\ell_m+1/2}{\ell}  \psi \right)    \sin({\psi}/{\ell}) d \psi \\
 &= \frac{1}{\ell^{2 }}   \int_{0}^{\frac {\ell \pi} 2} \left(\frac{{\psi}/{\ell}}{\sin({\psi}/{\ell})}  \right)^{\frac{q}{2}-1}  \prod_{m=1}^q   J_0\left( \frac{\ell_m+1/2}{l}  \psi \right)  \psi d \psi.
\end{align*}
For $\psi \in [0, \frac{\ell \pi}{2}]$,  we write $\left(\frac{\psi/\ell}{\sin( \psi/\ell)}  \right)^{\frac{q}2 -1 }=1+O( \frac{\psi^2}{\ell^2})$, that is
\begin{align*}
 A_{0,q}&= \frac{1}{\ell^{2 }}  \int_{0}^{\frac {\ell \pi} 2}   \prod_{m=1}^q  J_0\left( \frac{\ell_m+1/2}{\ell}  \psi \right)   \psi d \psi +   O \left( \frac{1}{\ell^4}   \int_{0}^{\frac {\ell \pi} 2}  \prod_{m=1}^q
 J_0\left( \frac{\ell_m+1/2}{\ell}  \psi \right) \;  \psi^3  d \psi  \right).
 \end{align*}
We consider now the error term. Since for $x \in [0,2]$ we have $J_0(x) \in (0,1]$, if $\varepsilon=\frac{2}{2+1/2}$, for  $\psi \in (0,\varepsilon]$ we have   $J_0\left( \frac{\ell_m+1/2}{\ell}  \psi \right) \in (0,1] $.  Recalling that $|J_0(x)|\le x^{- \frac 1 2} $, we have
\begin{align*}
& \int_{0}^{\frac {\ell \pi} 2} \prod_{m=1}^q |J_0\left( \frac{\ell_m+1/2}{\ell}  \psi \right)| \;  \psi^3  d \psi \\
&= \int_{0}^{\varepsilon}  \prod_{m=1}^q  |J_0\left( \frac{\ell_m+1/2}{\ell}  \psi \right)|  \;  \psi^3  d \psi + \prod_{m=1}^q \left( \frac{\ell_m +1/2}{\ell} \right)^{-\frac 1 2} \int_{\varepsilon}^{\frac{\ell \pi}{2}} \psi^{3-q/2} d \psi\\
 &\le \varepsilon^4+  \prod_{m=1}^q \left( \frac{\ell_m +1/2}{\ell} \right)^{-\frac 1 2} \times
  \begin{cases}
   \frac{1}{8(q-8)} \left( 16 \varepsilon^{4-\frac q 2}-  2^{\frac{q}{2}} (\ell \pi)^{4-\frac{q}{2}}  \right)        & \text{if } q \ne 8, \\
   \log(\frac{\ell \pi}{2})-\log(\varepsilon)        & \text{if } q=8,
  \end{cases}
\end{align*}
so that 
\begin{align*}
 A_{0,q}&= \frac{1}{\ell^{2 }}  \int_{0}^{\frac {\ell \pi} 2} \prod_{m=1}^q J_0\left( \frac{\ell_m+1/2}{\ell}  \psi \right)   \psi d \psi + \begin{cases}
   O \left( \ell^{-4}  + \ell^{-\frac{q}{2}} \right)       & \text{if } q \ne 8, \\
    O \left( \ell^{-4}  +  \ell^{-4}  \log(\frac{\ell \pi}{2})  \right)         & \text{if } q=8.
  \end{cases}
   \end{align*}
\begin{itemize}
\item For $A_{q,0}$, since, in view of Lemma \ref{hilb},  $\delta_m(\theta) \ll \theta^{\frac 1 2} \ell_m^{- \frac 3 2}$, we obtain
\end{itemize}
\begin{align*}
A_{q,0}&= \int_{0}^{\frac \pi 2}      \prod_{m'=1}^q \delta_{m'}(\theta) \sin \theta d \theta \ll \left(\frac{1}{2^{j-1}} \right)^{\frac 3 2 q} \int_0^{\frac \pi 2} \theta^{\frac q 2} \sin \theta d \theta  = O(\ell^{-\frac{3}{2}q}).
\end{align*}
\begin{itemize}
\item For $k=1, \dots, q-1$,
\end{itemize}
\begin{align*}
A_{q-k,k}& \ll \left( \frac{1}{2^{j-1}} \right)^{\frac 3 2 (q-k)}    \int_{0}^{\frac \pi 2} \theta^{\frac 1 2 (q-k)}  \left(\frac{\theta}{\sin \theta}  \right)^{\frac k 2}  \prod_{m=1}^k  J_0((\ell_m+1/2) \theta)  \sin \theta  d \theta\\
  & = \left( \frac{1}{2^{j-1}} \right)^{\frac 3 2 (q-k)}  \left( \frac{\pi}{2} \right)^{\frac 1 2 (q-k)}  A_{0,k}.
 \end{align*}
\end{proof}
\begin{remark}
Note that formula (\ref{PlJ0}) is meaningful only if  $\ell_1, \dots, \ell_q$ satisfy the following  `polygonal' conditions, i.e., for $q \ge 4$ and for all  $k=1, \dots,q$,  
\begin{align} \label{poly}
\ell_k \le \ell_1 + \cdots + \ell_{k-1} + \ell_{k+1} + \cdots + \ell_q, 
\end{align}
while otherwise we have  
$$\int_{0}^{\frac \pi 2}  P_{\ell_1}(\cos \theta) \cdots P_{\ell_q} (\cos \theta)   \sin \theta d \theta=0.$$
\end{remark}
\noindent We exploit Lemma \ref{lll} to prove the following:
\begin{lemma} \label{lem}
For $\ell=2^j$, $q > 4$ and $\tilde{\gamma}(\lfloor \ell x_k\rfloor,\ell)=b^2(\frac{\lfloor \ell x_k\rfloor}{\ell}) \frac{2 \lfloor \ell x_k\rfloor+1}{4 \pi \ell} \left( \frac{\lfloor \ell x_k\rfloor}{\ell}  \right)^{-\alpha} G(\lfloor \ell x_k\rfloor)$, we have
\begin{align*}
&  \lim_{\ell \to \infty} \ell^2  \int_{\frac 1 2}^2 \dots \int_{\frac 1 2}^2  \int_{0}^{\frac \pi 2}   \prod_{k=1}^q    \tilde{\gamma}(\lfloor \ell x_k\rfloor, \ell)    P_{\lfloor \ell x_k\rfloor}(\cos \theta) \sin \theta d \theta 
dx_1 \cdots d x_q\\
&\;\;=\left( \frac{G}{2 \pi} \right)^q  \int_{\frac 1 2}^2 \dots \int_{\frac 1 2}^2    \int_0^\infty  \prod_{k=1}^q b^2(x_k) x_k^{1-\alpha}  J_0(x_k \psi) \psi d \psi  
dx_1 \cdots d x_q.
\end{align*}
\end{lemma}
\begin{proof}
From Lemma \ref{lemma var}, we have
\begin{align*}
&   \int_{\frac 1 2}^2 \dots \int_{\frac 1 2}^2  \lim_{\ell \to \infty} \ell^2  \prod_{k=1}^q    \tilde{\gamma}(\lfloor \ell x_k\rfloor,\ell)   \int_{0}^{\frac \pi 2}   \prod_{k=1}^q  P_{\lfloor \ell x_k\rfloor}(\cos \theta) \sin \theta d \theta dx_1 \cdots d x_q\\
&\;\;= \int_{\frac 1 2}^2 \dots \int_{\frac 1 2}^2 \lim_{\ell \to \infty}   \prod_{k=1}^q   \tilde{\gamma}(\lfloor \ell x_k\rfloor, \ell)  \int_0^{\frac{\ell \pi }{2}}   \prod_{k=1}^q   J_0\left( \frac{\lfloor \ell x_k  \rfloor +1/2}{\ell}   \psi \right)  \psi d \psi 
dx_1 \cdots d x_q.
\end{align*}
Set $v(\ell,x_1, \dots,x_q)=\int_0^{\frac{\ell \pi }{2}}  \prod_{k=1}^q J_0\left( \frac{\lfloor \ell x_k  \rfloor +1/2}{\ell}   \psi \right)  \psi d \psi$, by dominated convergence we obtain that
\begin{align*}
&\lim_{\ell \to \infty}   v(\ell,x_1, \dots,x_q) \\
 &= \lim_{\ell \to \infty}  \int_{0}^\infty \prod_{k=1}^q J_0\left( \frac{\lfloor \ell x_k  \rfloor +1/2}{\ell}   \psi \right)  \psi d \psi- \lim_{\ell \to \infty}  \int_{2 \pi}^\infty \prod_{k=1}^q J_0\left( \frac{\lfloor \ell x_k  \rfloor +1/2}{\ell}   \psi \right)  \psi \ind_{ \{\psi \in [\frac{\ell \pi}{2},\infty)\}} d \psi\\
&=  \int_{0}^\infty  \prod_{k=1}^q J_0(x_k \psi) \psi d \psi
\end{align*}
in fact, there exists a finite real number $M$ such that
\begin{align*}
|  \prod_{k=1}^q J_0\left( \frac{\lfloor \ell x_k  \rfloor +1/2}{\ell}   \psi \right)  \psi \ind_{ \{\psi \in [\frac{\ell \pi}{2},\infty)\}}| &\le |  \prod_{k=1}^q J_0\left( \frac{\lfloor \ell x_k  \rfloor +1/2}{\ell}   \psi \right)  \psi | \\
& \le
\begin{cases}
\varepsilon & \text{if } \psi \in [0,\varepsilon] \\
\prod_{k=1}^q  \left( \frac{\ell}{\lfloor \ell x_k\rfloor + \frac 1 2}   \right)^{\frac 1 2} \psi^{1- \frac q 2}<M &\text{if } \psi \in [\varepsilon,\infty].
\end{cases}
\end{align*}
This leads to
\begin{align*}
 \lim_{\ell \to \infty}  \prod_{k=1}^q   \tilde{\gamma}(\lfloor \ell x_k\rfloor, \ell)  v(\ell,x_1, \dots,x_q)=\left( \frac{G}{2 \pi} \right)^q  \int_0^\infty  \prod_{k=1}^q b^2(x_k) x_k^{1-\alpha}  J_0(x_k \psi) \psi d \psi. 
 \end{align*}
On the other hand, we apply again dominated convergence to the sequence of measurable functions 
\begin{align*}
u_\ell(x_1, \dots,x_q)= \prod_{k=1}^q   \tilde{\gamma}(\lfloor \ell x_k\rfloor,\ell)
 \int_{0}^{\frac \pi 2}   \prod_{k=1}^q  P_{\lfloor \ell x_k\rfloor}(\cos \theta) \sin \theta d \theta
\end{align*}
on the set $[\frac 1 2, 2]^q$. Since, from Lemma \ref{lemma var}, for all $\ell$ and all $(x_1, \dots, x_q) \in [\frac 1 2, 2]^q$, we have 
\begin{align*}
|u_\ell(x_1, \dots,x_q)|&\le \prod_{k=1}^q | \tilde{\gamma}(\lfloor \ell x_k\rfloor,\ell)| \;\; | \int_{0}^{\frac \pi 2}   \prod_{k=1}^q  P_{\lfloor \ell x_k\rfloor}(\cos \theta)  \sin \theta d \theta | \\
&\le \prod_{k=1}^q | \tilde{\gamma}(\lfloor \ell x_k\rfloor,\ell)| \;\; | \int_0^{\frac{\ell \pi }{2}}  \prod_{k=1}^q J_0\left( \frac{\lfloor \ell x_k  \rfloor +1/2}{\ell}   \psi \right)  \psi d \psi +1| \\
&\le  \prod_{k=1}^q | \tilde{\gamma}(\lfloor \ell x_k\rfloor,\ell)| \left[  \int_{0}^{\varepsilon} \psi d \psi + \int_{\varepsilon}^{\frac{\ell \pi }{2}}  \prod_{k=1}^q \left( \frac{\ell}{\lfloor \ell x_k\rfloor + \frac 1 2}   \right)^{\frac 1 2}   \psi^{1- \frac q 2} d \psi +1\right],
\end{align*} 
where $\varepsilon= \frac{2}{ 2 + 1/ 2} $,  there exists a finite real number $M'$ such that for all $\ell$ and for all $(x_1, \dots, x_q) \in [\frac 1 2, 2]^q$
\begin{align*}
 \prod_{k=1}^q | \tilde{\gamma}(\lfloor \ell x_k\rfloor,\ell)| \left[  \varepsilon^2 +  \prod_{k=1}^q \left( \frac{\ell}{\lfloor \ell x_k\rfloor + \frac 1 2}   \right)^{\frac 1 2}  \frac{4 \varepsilon^{2-\frac q 2}-2^{\frac q 2} (\ell \pi)^{2 -\frac q 2}}{2(q-4)} +1\right] \le M'
\end{align*}
and this leads to
\begin{align*}
\int_{\frac 1 2}^{2}  \cdots  \int_{\frac 1 2}^{2}   \lim_{\ell \to \infty}   u_\ell(x_1, \dots,x_q)  d x_1 \cdots d x_q=\lim_{\ell \to \infty}  \int_{\frac 1 2}^2 \dots \int_{\frac 1 2}^2 u_\ell(x_1, \dots,x_q)  d x_1 \cdots d x_q. 
\end{align*}
\end{proof}

\begin{remark} The previous discussion yields the following corollary: for $q>4$, $\ell=2^j$, $x_k \in [\frac1 2,2]$ with $k=1,\dots,q$, we have
\begin{align*} 
\lim_{\ell \to \infty} \ell^2 \int_0^{\frac \pi 2}  P_{\lfloor \ell x_1  \rfloor} (\cos \theta)  \cdots  P_{\lfloor 
\ell x_q  \rfloor} (\cos \theta) \sin \theta d \theta =  \int_{0}^\infty  J_0(x_1 \psi) \cdots  J_0(x_q \psi) \psi d \psi.
\end{align*}
For $q=3$ it is well-known that, if $x_1,x_2,x_3 >0$, we have
\begin{align*}
\int_0^\infty J_0(x_1 \psi)   J_0(x_2 \psi)  J_0(x_3 \psi) \psi d \psi= \begin{cases} 
\frac{1}{ 2 \pi  \Delta }, &{\text if }\;\; |x_1-x_2| < x_3 < x_1+x_2,\\
0, &{\text if }\;\; 0<x_3 \le |x_1-x_2|  \; \; {\text or }\;\; x_3 \ge x_1+x_2, 
\end{cases}
\end{align*}
where $\Delta=\frac 1 4  \sqrt{[x_3^2-(x_1-x_2)^2] [(x_1+x_2)^2-x_3^2]}$ is equal to the area of a triangle whose sides are $x_1$, $x_2$ and $x_3$, see \cite{GR} formula 6.578.9.
\end{remark}
Before proving Theorem \ref{2}, we introduce some further notation i.e.
$$B_\ell=\sum_{\ell_1=2^{j-1}}^{2^{j+1}} b^2 \left( \frac{\ell_1}{\ell} \right) \frac{2 \ell_1+1}{4 \pi} \ell_1^{-\alpha} G(\ell_1),$$
and we prove the last lemma:
\begin{lemma} \label{4} For $\ell=2^j$, we have that
\begin{align*}
\lim_{\ell \to \infty} \ell^{ \alpha-2} B_\ell=\frac{G}{2 \pi} \int_{\frac 1 2}^{2} b^2(x) x^{1-\alpha} d x.
\end{align*}
\end{lemma}
\begin{proof}
We first note that
\begin{align*}
\lim_{\ell \to \infty} \ell^{ \alpha-2} B_\ell&=\lim_{\ell \to \infty} \frac{\ell}{\ell^{2-\alpha}} \sum_{\ell_1=\frac \ell 2}^{2 \ell} \int_{\frac{\ell_1}{\ell}}^{\frac{\ell_1+1}{\ell}} b^2\left(\frac{\lfloor \ell x \rfloor}{\ell}\right) \lfloor \ell x \rfloor^{-\alpha} \frac{ 2 \lfloor \ell x \rfloor+1}{4 \pi } G(\lfloor \ell x \rfloor)d x\\
&=\lim_{\ell \to \infty} \int_{\frac 1 2}^{2} b^2\left(\frac{\lfloor \ell x \rfloor}{\ell}\right) \left( \frac{\lfloor \ell x \rfloor}{\ell}\right)^{-\alpha} \frac{ 2 \lfloor \ell x \rfloor+1}{2  \ell} \frac{G(\lfloor \ell x \rfloor)}{2 \pi}  d x,
\end{align*}
and using dominated convergence, we have the statement.
\end{proof}
\begin{proof}[Proof of Theorem \ref{2}]
\begin{align*}
Var[\nu_{j;q}]&=\mathbb{E} \left[\left( \int_{S^2}
H_q(\tilde{\beta}_j(x)) d x \right)^2\right]=\int_{S^2 \times S^2}
\mathbb{E} \left[ H_q(\tilde{\beta}_j(x_1))
H_q(\tilde{\beta}_j(x_2)) \right] d \sigma(x_1) d\sigma(x_2)
\end{align*}
by Proposition \ref{j-g}, for $\ell=2^j$, we have 
\begin{align*}
Var[\nu_{j;q}]&=q! \int_{S^2 \times S^2} \left\{ \mathbb{E} \left[ \tilde{\beta}_j(x_1) \tilde{\beta}_j(x_2) \right] \right\}^q d \sigma(x_1) d \sigma(x_2)\\
&=q! B_\ell^{-q} \int_{S^2\times S^2} \left\{ \mathbb{E} \left[ {\beta}_j(x_1) {\beta}_j(x_2) \right] \right\}^q d \sigma(x_1) d \sigma(x_2) \\
&= q! B_\ell^{-q} \int_{S^2 \times S^2} \left\{ \sum_{\ell_1=2^{j-1}}^{2^{j+1}} b^2\left(\frac{\ell_1}{\ell} \right)  \frac{2 \ell_1+1}{4 \pi}  \ell_1^{-\alpha} G(\ell_1)  P_{\ell_1}(\langle x_1,x_2 \rangle)  \right\}^q d \sigma(x_1) d \sigma(x_2).
\end{align*}
Let  $\tilde{\gamma}(\ell_k, \ell):=b^2\left(\frac{\ell_k}{\ell}\right)   \frac{2 \ell_k+1}{4 \pi \ell}  \left(\frac{\ell_k}{\ell}\right)^{-\alpha} G(\ell_k)$ where, for all $k=1, \dots, q$, $\ell_k \in [2^{j-1},2^{j+1}]$; we have
\begin{align*}
Var[\nu_{j;q}]&= q! \ell^{- \alpha q+q} B_\ell^{-q} \sum_{\ell_1 \dots
\ell_q} \tilde{\gamma}(\ell_1,\ell) \cdots \tilde{\gamma}(\ell_q,\ell) \\
& \;\;\; \times \int_{S^2 \times S^2} P_{\ell_1}(\langle x_1 ,x_2 \rangle) \cdots
P_{\ell_q} (\langle x_1 ,x_2 \rangle) d \sigma(x_1) d \sigma(x_2),
\end{align*}
where 
\begin{align*}
\int_{S^2 \times S^2} P_{\ell_1}(\langle x_1 ,x_2 \rangle) \cdots P_{\ell_q} (\langle x_1 ,x_2 \rangle) d \sigma(x_1) d \sigma(x_2)&= 8 \pi^2 \int_{0}^\pi  P_{\ell_1}(\cos \theta) \cdots P_{\ell_q} (\cos \theta)  \sin \theta d \theta
\end{align*}
Then
\begin{align*}
Var[\nu_{j;q}]&=q! 8 \pi^2  \ell^{- \alpha q+q}B_\ell^{-q}   \sum_{\ell_1 \dots \ell_q} \tilde{\gamma}(\ell_1,\ell) \cdots \tilde{\gamma}(\ell_q,\ell)  \int_{0}^{ \pi }  P_{\ell_1}(\cos \theta) \cdots P_{\ell_q} (\cos \theta)  \sin \theta d \theta \\
&= q! 8 \pi^2  \ell^{- \alpha q+q}B_\ell^{-q}    \sum_{\ell_1 \dots \ell_q \atop \sum l_k  \text{even}} \tilde{\gamma}(\ell_1,\ell) \cdots \tilde{\gamma}(\ell_q,\ell) \; 2  \int_{0}^{\frac \pi 2 }  P_{\ell_1}(\cos \theta) \cdots P_{\ell_q} (\cos \theta)  \sin \theta d \theta 
\end{align*}
since
\begin{align*}
&\int_{0}^\pi  P_{\ell_1}(\cos \theta) \cdots P_{\ell_q} (\cos \theta) (\cos \theta)  \sin \theta d \theta\\
&\hspace{0.4cm}=\begin{cases}
2 \int_{0}^{\frac \pi 2} P_{\ell_1}(\cos \theta) \cdots P_{\ell_q} (\cos \theta)  \sin \theta d \theta, &\text{for }\sum_{k=1}^q l_k  \text{ even},\\
0, &\text{for }\sum_{k=1}^q l_k  \text{ odd}.
\end{cases}
\end{align*}
Also 
\begin{align*}
 Var[\nu_{j;q}] &= q! 8 \pi^2  \ell^{- \alpha q+q}B_\ell^{-q}    \sum_{\ell_1 \dots \ell_q} \tilde{\gamma}(\ell_1,\ell) \cdots \tilde{\gamma}(\ell_q,\ell)  \int_{0}^{\frac \pi 2 }  P_{\ell_1}(\cos \theta) \cdots P_{\ell_q} (\cos \theta)  \sin \theta d \theta \\
&=q! 8 \pi^2  \ell^{- \alpha q+2q} B_\ell^{-q}   \sum_{\ell_1= \frac \ell 2}^{2 \ell} \cdots \sum_{\ell_q=\frac \ell 2}^{2 \ell} \int_{\frac{ \ell_1}{\ell}}^{\frac{\ell_1+1}{\ell}}  \dots   \int_{\frac{\ell_q}{\ell}}^{\frac{\ell_q+1}{\ell}}    \tilde{\gamma}(\lfloor \ell x_1\rfloor, \ell) \cdots \tilde{\gamma}(\lfloor \ell x_q\rfloor, \ell) \\
&\;\;\; \times  \int_{0}^{ \pi }  P_{\lfloor \ell x_1\rfloor}(\cos \theta) \cdots P_{\lfloor \ell x_q\rfloor} (\cos \theta)  \sin \theta d \theta  dx_1 \cdots d x_q\\
&=q! 8 \pi^2  \ell^{- \alpha q+2q} B_\ell^{-q}    \int_{\frac 1 2}^{\frac{2 \ell+1}{\ell}}  \dots    \int_{\frac 1 2}^{\frac{2 \ell+1}{\ell}}   \tilde{\gamma}(\lfloor \ell x_1\rfloor, \ell) \cdots \tilde{\gamma}(\lfloor \ell x_q\rfloor, \ell) \\
&\;\;\; \times  \int_{0}^{ \pi }  P_{\lfloor \ell x_1\rfloor}(\cos \theta) \cdots P_{\lfloor \ell x_q\rfloor} (\cos \theta)  \sin \theta d \theta  dx_1 \cdots d x_q
\end{align*}
and then
\begin{align*}
\lim_{\ell \to \infty}\ell^2 Var[\nu_{j;q}] &= \lim_{\ell \to \infty} q! 8 \pi^2 \ell^{- \alpha q+2q} B_\ell^{-q}   \ell^2  \int_{\frac 1 2}^{2} \cdots  \int_{\frac 1 2}^{2}   \tilde{\gamma}(\lfloor \ell x_1\rfloor,\ell) \cdots \tilde{\gamma}(\lfloor \ell x_q\rfloor, \ell) \\
& \;\;\; \times  \int_{0}^{\pi }  P_{\lfloor \ell x_1\rfloor}(\cos \theta) \cdots P_{\lfloor \ell x_q\rfloor} (\cos \theta)  \sin \theta d \theta dx_1 \cdots d x_q.
\end{align*}
The statement follows by applying Lemma \ref{lem} and Lemma \ref{4}.
\end{proof}
For the cases $q=2,3,4$ we write a different proof based on the
representation of the integral of the product of spherical
harmonics in terms of Wigner's 3j coefficients.

\begin{theorem}
For $q =2$,  we have
$$\lim_{j \to \infty}2^{2j}{\rm Var}[\nu_{j;2}]= 2! c_2$$
where $$c_2= \frac{8 \pi^2 }{\left( \int_{\frac 1 2}^{2} b^2(x) x^{1- \alpha} d x \right)^2}  \int_{\frac 1 2}^{2}   b^4(x_1) x_1^{1-2 \alpha}  d x_1.$$
\end{theorem}
\begin{proof}
For  $\ell=2^j$ and $\ell_1,\ell_2 \in [2^{j-1},2^{j+1}]$ we have as before
\begin{align*}
Var[\nu_{j;2}]&=2! B_\ell^{-2} \int_{S^2\times S^2} \left\{ \mathbb{E} \left[ {\beta}_j(x_1) {\beta}_j(x_2) \right] \right\}^2 d \sigma(x_1) d \sigma(x_2) \\
&= 2! 8 \pi^2 \ell^{-2\alpha+2}B_\ell^{-2} \sum_{\ell_1 \ell_2} b^2\left(\frac{\ell_1}{\ell} \right) b^2\left(\frac{\ell_2}{\ell} \right)  \frac{2 \ell_1+1}{4 \pi \ell}  \frac{2 \ell_2+1}{4 \pi \ell}  \left(\frac{\ell_1 \ell_2}{\ell^2}\right)^{-\alpha} G(\ell_1)  G(\ell_2) \\
&\hspace{0.4cm} \times \int_{0}^\pi  P_{\ell_1}(\cos \theta) P_{\ell_2} (\cos \theta)  \sin \theta d \theta,
\end{align*}
from the orthogonality property of Legendre polynomials, we have
\begin{align*}
Var[\nu_{j;2}]&= 2! 8 \pi^2 \ell^{-2 \alpha+2} B_\ell^{-2} \sum_{\ell_1=\frac \ell 2}^{2 \ell} b^4\left(\frac{\ell_1}{\ell} \right)  \left(\frac{2 \ell_1+1}{4 \pi \ell}\right)^2  \left( \frac{\ell_1}{\ell}\right)^{-2\alpha} G^2(\ell_1)\; \frac{2}{2 \ell_1+1}\\
&= 2! 8 \pi^2 \ell^{-2 \alpha+2} B_\ell^{-2}   \int_{\frac{1}{2}}^{\frac{2 \ell +1}{\ell}}  b^4\left(\frac{\lfloor \ell x_1\rfloor }{\ell} \right)  \frac{2\lfloor \ell x_1\rfloor +1}{2 \ell}   \left( \frac{\lfloor \ell x_1\rfloor}{\ell}\right)^{-2\alpha} \left(\frac{G(\lfloor \ell x_1\rfloor )}{2 \pi}\right)^2  d x_1.
\end{align*}
So we see that
\begin{align*}
&\lim_{\ell \to \infty} \ell^2 Var[\nu_{j;2}]\\
&\hspace{0.4cm}=\lim_{\ell \to \infty} 2! 8 \pi^2 \ell^{-2 \alpha+2} B_\ell^{-2} \ell^2 \int_{\frac{1}{2}}^{2}  b^4\left(\frac{\lfloor \ell x_1\rfloor }{\ell} \right)  \frac{2\lfloor \ell x_1\rfloor +1}{2 \ell}   \left( \frac{\lfloor \ell x_1\rfloor}{\ell}\right)^{-2\alpha} \left(\frac{G(\lfloor \ell x_1\rfloor )}{2 \pi}\right)^2  d x_1
\end{align*}
and by applying Lemma \ref{4} and dominated convergence we arrive at the statement.
\end{proof}

We introduce now the Wigner's $3j$ coefficients
$$\left( \begin{matrix}  \ell_1 & \ell_2 &\ell_3 \\   m_1 & m_2 &m_3  \end{matrix}  \right), \hspace{1cm} -(2 \ell_i+1) \le m_i \le 2 \ell_i+1,\;\:i =1,2,3.$$
The Wigner's $3j$ coefficients are zero unless the triangle conditions $| \ell_i- \ell_r|\le \ell_k \le \ell_i+ \ell_r$ for $i,r,k=1,2,3$ are satisfied and $m_1+m_2+m_3=0$, see \cite{marpecbook}  Section 3.5.3 for further details. When $m_1=m_2=m_3=0$, the analytic expression reduces to
\begin{align} \label{mmmm}
 \left( \begin{matrix}  \ell_1 & \ell_2 & \ell_3 \\   0 & 0 &0  \end{matrix}  \right)
\end{align}
\begin{align}
= \begin{cases}
 \frac{  (-1)^{\frac{\ell_1+\ell_2-\ell_3}{2}} \left(  \frac{\ell_1+\ell_2+\ell_3}{2} \right)!}{\left(  \frac{\ell_1+\ell_2-\ell_3}{2} \right)! \left(  \frac{\ell_1-\ell_2+\ell_3}{2} \right)! \left(  \frac{-\ell_1+\ell_2+\ell_3}{2} \right)!}  \left[ \frac{\left( \ell_1+\ell_2-\ell_3 \right)! \left( \ell_1-\ell_2+\ell_3 \right)! \left( -\ell_1+\ell_2+\ell_3\right)!}{\left(  \ell_1+\ell_2+\ell_3+1 \right)!}  \right]^{\frac 1 2},
 & \ell_1+\ell_2+\ell_3 \text{\;\;even},\\  0, &  \ell_1+\ell_2+\ell_3 \text{\;\;odd}, \end{cases} \nonumber
\end{align}
see \cite{VMK}, equations 8.1.2.12 and 8.5.2.32.
\begin{lemma} \label{wigner}
For every fixed $(x_1,x_2,x_3) \in P_3$, define 
$$g_\ell (x_1,x_2,x_3)= \left( \begin{matrix}   \lfloor \ell x_1 \rfloor &    \lfloor \ell x_2 \rfloor   &   \lfloor \ell x_3 \rfloor   \\   0 & 0 &0  \end{matrix}  \right)^2,$$ 
we have that 
$$\lim_{\ell \to \infty} \ell^2 g_{\ell}(x_1,x_2,x_3)=\frac{2}{ \pi}  \frac{1}{\sqrt{  x_1  +   x_2 -  x_3 } \sqrt{  x_1  -   x_2 +  x_3 } \sqrt{-  x_1  +   x_2 +  x_3 } \sqrt{x_1  +   x_2 +  x_3 }},$$
where the limit is defined for all $\ell$ such that $\lfloor \ell x_1 \rfloor +\lfloor \ell x_2 \rfloor+\lfloor \ell x_3 \rfloor$ is even.
\end{lemma}
\begin{proof}
Let  $\lambda_0=\lfloor \ell x_1 \rfloor+ \lfloor \ell x_ 2 \rfloor+\lfloor \ell x_3\rfloor$, $\lambda_1=- \lfloor \ell x_1\rfloor+\lfloor \ell x_2 \rfloor+\lfloor \ell x_3 \rfloor$, $\lambda_2=\lfloor \ell x_1 \rfloor-\lfloor \ell x_2 \rfloor+\lfloor \ell x_3 \rfloor$ and $\lambda_3=\lfloor \ell x_1 \rfloor+\lfloor \ell x_2 \rfloor-\lfloor \ell x_3 \rfloor$,   from (\ref{mmmm}), by applying Stirling's formula
$$\ell!=\sqrt{2 \pi} \ell^{\ell+\frac 1 2} e^{-\ell}+O(\ell^{-1})$$
we see that
\begin{align*} 
&\lim_{\ell \to \infty} \ell^2 g_\ell (x_1,x_2,x_3)\\
&=\lim_{\ell \to \infty} \ell^2 \left[  \frac{\sqrt{2 \pi} \left(\frac{\lambda_0}{2}  \right)^{\frac{\lambda_0}{2}+ \frac 1 2} e^{-\frac{\lambda_0}{2}} }{\prod_{i=1}^3  \sqrt{2 \pi} \left(\frac{\lambda_i}{2}  \right)^{\frac{\lambda_i}{2}+ \frac 1 2} e^{-\frac{\lambda_i}{2}}     } \right]^2 \frac{\prod_{i=1}^3  \sqrt{2 \pi} \lambda_i ^{\lambda_i+ \frac 1 2} e^{-\lambda_i}     }{\sqrt{2 \pi} (\lambda_0+1)^{\lambda_0+ \frac 3 2} e^{-\lambda_0-1}}  \\
&=\lim_{\ell \to \infty}  \ell^2 e   \frac{2 \pi (2 \pi)^{\frac 3 2}}{(2 \pi)^3 \sqrt{2 \pi}} 2^{- \lambda_0+2+\sum_{i=1}^3 \lambda_i} \frac{\lambda_0^{\lambda_0+1}  \prod_{i=1}^3 \lambda_i^{\lambda_i+ \frac 1 2}}{(\lambda_0+1)^{\lambda_0+ \frac 3 2}  \prod_{i=1}^3 \lambda_i^{\lambda_i+1}}\\
&= \lim_{\ell \to \infty} \frac{2 e}{ \pi} \frac{\ell^2}{\sqrt{\lambda_1 \lambda_2 \lambda_3}} \frac{\lambda_0}{(\lambda_0+1)^{\frac 3 2}} \left( 1+ \frac{1}{\lambda_0} \right)^{-\lambda_0}\\
&= \lim_{\ell \to \infty} \frac{2 e}{\pi}  \frac{\ell^2}{\sqrt{(-\lfloor \ell x_1\rfloor+\lfloor \ell x_2\rfloor+\lfloor \ell x_3\rfloor)(\lfloor \ell x_1\rfloor-\lfloor \ell x_2\rfloor+\lfloor \ell x_3\rfloor) (\lfloor \ell x_1\rfloor+\lfloor \ell x_2\rfloor-\lfloor \ell x_3\rfloor)}} \\
&\hspace{0.4cm} \times  \frac{\lfloor \ell x_1\rfloor+\lfloor \ell x_2\rfloor+\lfloor \ell x_3\rfloor}{(\lfloor \ell x_1\rfloor+\lfloor \ell x_2\rfloor+\lfloor \ell x_3\rfloor+1)^{\frac 3 2}}  \left( 1+\frac{1}{\lfloor \ell x_1\rfloor+\lfloor \ell x_2\rfloor+\lfloor \ell x_3\rfloor}  \right)^{-(\lfloor \ell x_1\rfloor+\lfloor \ell x_2\rfloor+\lfloor \ell x_3\rfloor)}\\
&=\frac{2}{ \pi}  \frac{1}{\sqrt{  x_1  +   x_2 -  x_3 } \sqrt{  x_1  -   x_2 +  x_3 } \sqrt{-  x_1  +   x_2 +  x_3 } \sqrt{x_1  +   x_2 +  x_3 }}.
\end{align*}
\end{proof}
\begin{remark}
Note that for $\lfloor \ell x_1\rfloor=\lfloor \ell x_2\rfloor=\lfloor \ell x_3\rfloor=\ell$ we have the same result as in \cite{MaWi22} Lemma A.1, in fact 
\begin{align*}
\lim_{\ell \to \infty} \ell^2 \left( \begin{matrix}  \ell &  \ell   &\ell   \\   0 & 0 &0  \end{matrix}  \right)^2= \lim_{\ell \to \infty} \frac{2 e}{\pi}  \frac{\ell^2}{\sqrt{\ell^3}}  \frac{3 \ell}{(3 \ell+1)^{\frac 3 2}}  \left( 1+\frac{1}{3 \ell}  \right)^{-3 \ell}=\frac{2}{\pi \sqrt 3}.
\end{align*}
\end{remark}

\begin{theorem}
For $q =3$,  we have
$$\lim_{j \to \infty}2^{2j}{\rm Var}[\nu_{j;3}]= 3! c_3$$
where
\begin{align*}c_3&= \frac{16 \pi}{  \left(\int_{\frac 1 2}^{2} b^2(x) x^{1- \alpha} d x \right)^3} \int_{\frac 1 2}^2 \cdots \int_{\frac 1 2}^2 \prod_{i=1}^3 b^2(x_i)  x_i^{1-\alpha}  \\
&\hspace{0.4cm} \times    \frac{1}{\sqrt{  x_1  +   x_2 -  x_3 } \sqrt{  x_1  -   x_2 +  x_3 } \sqrt{-  x_1  +   x_2 +  x_3 } \sqrt{x_1  +   x_2 +  x_3 }} \\
&\hspace{0.4cm} \times  \ind_{P_3}(x_1,x_2,x_3) d x_1 d x_2 d x_3.\end{align*}
\end{theorem}
\begin{proof}
For  $\ell=2^j$ and $\ell_1,\ell_2, \ell_3 \in [2^{j-1},2^{j+1}]$ we have
\begin{align*}
Var[\nu_{j;3}]&=3! B_\ell^{-3} \int_{S^2\times S^2} \left\{ \mathbb{E} \left[ {\beta}_j(x_1) {\beta}_j(x_2) \right] \right\}^3 d \sigma(x_1) d \sigma(x_2) \\
&= 3! 8 \pi^2 B_\ell^{-3} \sum_{\ell_1 \ell_2 \ell_3}  \prod_{i=1}^3 b^2\left(\frac{\ell_i}{\ell} \right) \frac{2 \ell_i+1}{4 \pi}    \ell_i^{-\alpha} G(\ell_i)  \int_{0}^\pi  P_{\ell_1}(\cos \theta) P_{\ell_2} (\cos \theta)  P_{\ell_3}(\cos \theta)  \sin \theta d \theta.
\end{align*}
By expressing Legendre polynomials in terms of spherical harmonics 
and by applying the well-known formula for the integral of the product of three spherical harmonics over the sphere 
(see \cite{marpecbook} Section 3.5.3 for a proof), we obtain
$$\int_{0}^\pi  P_{\ell_1}(\cos \theta) P_{\ell_2} (\cos \theta)  P_{\ell_3}(\cos \theta)  \sin \theta d \theta=2 \left( \begin{matrix}  \ell_1 & \ell_2 &\ell_3 \\   0 & 0 &0  \end{matrix}  \right)^2,$$
and then, from (\ref{mmmm}), 
\begin{align*}
Var[\nu_{j;3}]&= 3! 8 \pi^2 B_\ell^{-3} \sum_{\ell_1 \ell_2 \ell_3 \atop \sum l_k \text{even}}  \prod_{i=1}^3  b^2\left(\frac{\ell_i}{\ell} \right)   \frac{2 \ell_i+1}{4 \pi}  \ell_i^{-\alpha} G(\ell_i)  \; 2 \left( \begin{matrix}  \ell_1 & \ell_2 &\ell_3 \\   0 & 0 &0  \end{matrix}  \right)^2\\
&= 3! 8 \pi^2  \ell^{-3 \alpha+6}  B_\ell^{-3} \sum_{\ell_1 \ell_2 \ell_3 \atop \sum l_k \text{even}} \int_{\frac{l_1} l}^{\frac{l_1+1}{l}} \cdots \int_{\frac{l_3} l}^{\frac{l_3+1}{l}}   \prod_{i=1}^3     b^2\left(\frac{\lfloor \ell x_i\rfloor }{\ell} \right)    \frac{2\lfloor \ell x_i\rfloor+1}{2 \ell}   \\
&\hspace{0.4cm} \times \left( \frac{\lfloor \ell x_i \rfloor }{\ell} \right)^{-\alpha} \frac{G(\lfloor \ell x_i\rfloor)}{2 \pi} 2  \left( \begin{matrix}  \lfloor \ell x_1\rfloor & \lfloor \ell x_2\rfloor &\lfloor \ell x_3\rfloor \\   0 & 0 &0  \end{matrix}  \right)^2 d x_1 d x_2 d x_3.
\end{align*}
Applying dominated convergence again and Lemma \ref{wigner},  
\begin{align*}
& \hspace{3cm} \lim_{\ell \to \infty} \ell^2 Var[\nu_{j;3}]\\
&=\lim_{\ell \to \infty} 3! 8 \pi^2  \ell^{-3 \alpha+6}  B_\ell^{-3} \sum_{\ell_1 \ell_2 \ell_3 \atop \sum l_k \text{even}} \int_{\frac{l_1} l}^{\frac{l_1+1}{l}} \cdots \int_{\frac{l_3} l}^{\frac{l_3+1}{l}}   \prod_{i=1}^3     b^2\left(\frac{\lfloor \ell x_i\rfloor }{\ell} \right)    \frac{2\lfloor \ell x_i\rfloor+1}{2 \ell}  \left( \frac{\lfloor \ell x_i \rfloor }{\ell} \right)^{-\alpha} \\
&\hspace{0.4cm} \times  \frac{G(\lfloor \ell x_i\rfloor)}{2 \pi} \frac{2 e}{\pi} 2 \frac{\ell^2}{\sqrt{(-\lfloor \ell x_1\rfloor+\lfloor \ell x_2\rfloor+\lfloor \ell x_3\rfloor)(\lfloor \ell x_1\rfloor-\lfloor \ell x_2\rfloor+\lfloor \ell x_3\rfloor) (\lfloor \ell x_1\rfloor+\lfloor \ell x_2\rfloor-\lfloor \ell x_3\rfloor)}} \\
&\hspace{0.4cm} \times  \frac{\lfloor \ell x_1\rfloor+\lfloor \ell x_2\rfloor+\lfloor \ell x_3\rfloor}{(\lfloor \ell x_1\rfloor+\lfloor \ell x_2\rfloor+\lfloor \ell x_3\rfloor+1)^{\frac 3 2}}  \left( 1+\frac{1}{\lfloor \ell x_1\rfloor+\lfloor \ell x_2\rfloor+\lfloor \ell x_3\rfloor}  \right)^{-(\lfloor \ell x_1\rfloor+\lfloor \ell x_2\rfloor+\lfloor \ell x_3\rfloor)}\\
&\hspace{0.4cm} \times \ind_{P_3}(x_1,x_2,x_3) d x_1 d x_2 d x_3\\
&=\lim_{\ell \to \infty} 3! 8 \pi^2  \ell^{-3 \alpha+6}  B_\ell^{-3}  \int_{\frac 1 2}^{\frac{2 l+1}{l}} \cdots \int_{\frac 1 2}^{\frac{ 2 l+1}{l}}   \prod_{i=1}^3     b^2\left(\frac{\lfloor \ell x_i\rfloor }{\ell} \right)    \frac{2\lfloor \ell x_i\rfloor+1}{2 \ell}  \left( \frac{\lfloor \ell x_i \rfloor }{\ell} \right)^{-\alpha} \\
&\hspace{0.4cm} \times  \frac{G(\lfloor \ell x_i\rfloor)}{2 \pi} \frac{2 e}{\pi} \frac{\ell^2}{\sqrt{(-\lfloor \ell x_1\rfloor+\lfloor \ell x_2\rfloor+\lfloor \ell x_3\rfloor)(\lfloor \ell x_1\rfloor-\lfloor \ell x_2\rfloor+\lfloor \ell x_3\rfloor) (\lfloor \ell x_1\rfloor+\lfloor \ell x_2\rfloor-\lfloor \ell x_3\rfloor)}} \\
&\hspace{0.4cm} \times  \frac{\lfloor \ell x_1\rfloor+\lfloor \ell x_2\rfloor+\lfloor \ell x_3\rfloor}{(\lfloor \ell x_1\rfloor+\lfloor \ell x_2\rfloor+\lfloor \ell x_3\rfloor+1)^{\frac 3 2}}  \left( 1+\frac{1}{\lfloor \ell x_1\rfloor+\lfloor \ell x_2\rfloor+\lfloor \ell x_3\rfloor}  \right)^{-(\lfloor \ell x_1\rfloor+\lfloor \ell x_2\rfloor+\lfloor \ell x_3\rfloor)}\\
&\hspace{0.4cm} \times \ind_{P_3}(x_1,x_2,x_3) d x_1 d x_2 d x_3
\end{align*}
Then, by dominated convergence again and Lemma \ref{4}, we arrive at the statement.
\end{proof}
\begin{theorem}
For $q =4$,  
$$\lim_{j \to \infty}2^{2j}{\rm Var}[\nu_{j;4}]= 4! c_4$$
where
\begin{align*}c_4&=  \frac{32}{  \left(\int_{\frac 1 2}^{2} b^2(x) x^{1- \alpha} d x \right)^4}
\int_{\frac{1}{2}}^{2} \cdots \int_{\frac{1}{2}}^{2}  \prod_{i=1}^4   b^2(x_i) x_i^{1-\alpha}   \\
& \hspace{0.4cm} \times  \int_{0}^{4} y    \frac{1}{\sqrt{-x_1+x_2+ y } \sqrt{x_1-x_2+ y } \sqrt{x_1+x_2- y } \sqrt{x_1+x_2+ y }} \\
&\hspace{0.4cm} \times \frac{1}{\sqrt{-x_3+x_4+ y } \sqrt{x_3-x_4+ y } \sqrt{x_3+x_4- y } \sqrt{x_3+x_4+ y }}\\
& \hspace{0.4cm} \times \ind_{P_3}(x_1,x_2, y) \ind_{P_3}(y,x_3,x_4)   \;  dy \; d x_1\cdots d x_4.
\end{align*}
\end{theorem}
\begin{proof}
For  $\ell=2^j$ and $\ell_1, \ell_2, \ell_3, \ell_4 \in [2^{j-1},2^{j+1}]$ we have
\begin{align*}
Var[\nu_{j;4}]&=4! B_\ell^{-4} \int_{S^2\times S^2} \left\{ \mathbb{E} \left[ {\beta}_j(x_1) {\beta}_j(x_2) \right] \right\}^4 d \sigma(x_1) d \sigma(x_2) \\
&= 4! 8 \pi^2 B_\ell^{-4} \sum_{\ell_1 \ell_2 \ell_3 \ell_4} \prod_{i=1}^4 b^2\left(\frac{\ell_i}{\ell} \right)    \frac{2 \ell_i+1}{4 \pi}  \ell_i^{-\alpha} G(\ell_i) \int_{0}^\pi  \prod_{i=1}^4 P_{\ell_i}(\cos \theta) \sin \theta d \theta.
\end{align*}
From the product formula
\begin{align*}
Y_{\ell_10}(\theta, \phi)  Y_{\ell_20}(\theta, \phi) =\sqrt{\frac{(2 \ell_1+1)(2 \ell_2+1)}{4 \pi}} \sum_{L=|\ell_1-\ell_2|}^{\ell_1+\ell_2} \sqrt{2L+1}   \left( \begin{matrix}  \ell_1 & \ell_2 &L \\   0 & 0 &0  \end{matrix}  \right)^2 Y_{L0}(\theta,\phi),
\end{align*}
and the orthogonality property of spherical harmonics, we obtain the following formula for the integral of the product of four spherical harmonics over the sphere
\begin{align*}
&\int_{0}^{2 \pi} \int_{0}^\pi \prod_{i=1}^4 Y_{\ell_i 0}(\theta, \phi) \sin \theta d \theta d \phi\\
&=\sqrt{\frac{(2 \ell_1+1)(2 \ell_2+1)}{4 \pi}}  \sqrt{\frac{(2 \ell_3+1)(2 \ell_4+1)}{4 \pi}} \sum_{L_1=|\ell_1-\ell_2|}^{\ell_1+\ell_2} \sqrt{2 L_1+1} \left( \begin{matrix}  \ell_1 & \ell_2 &L_1 \\   0 & 0 &0  \end{matrix}  \right)^2 \\
& \hspace{0.4cm} \times \sum_{L_2=|\ell_3-\ell_4|}^{\ell_3+\ell_4} \sqrt{2 L_2+1} \left( \begin{matrix}  \ell_3 & \ell_4 &L_2 \\   0 & 0 &0  \end{matrix}  \right)^2 \delta_{L_1}^{L_2} \\
&=4 \pi\prod_{i=1}^4 \sqrt{\frac{2 \ell_i+1}{4 \pi}} \sum_{L} (2L+1) \left( \begin{matrix}  \ell_1 & \ell_2 &L \\   0 & 0 &0  \end{matrix}  \right)^2 \left( \begin{matrix}  \ell_3 & \ell_4 &L \\   0 & 0 &0  \end{matrix}  \right)^2,
\end{align*}
that is 
\begin{align*}
\int_{0}^\pi  \prod_{i=1}^4 P_{\ell_i}(\cos \theta) \sin \theta d \theta=2 \sum_{L} (2L+1) \left( \begin{matrix}  \ell_1 & \ell_2 &L \\   0 & 0 &0  \end{matrix}  \right)^2 \left( \begin{matrix}  \ell_3 & \ell_4 &L \\   0 & 0 &0  \end{matrix}  \right)^2.
\end{align*}
We can write the variance as
\begin{align*}
Var[\nu_{j;4}]&=4! 8 \pi^2 B_\ell^{-4} \sum_{l_1 l_2 l_3 l_4} \; \prod_{i=1}^4 b^2\left(\frac{\ell_i}{\ell} \right)    \frac{2 \ell_i+1}{4 \pi}  \ell_i^{-\alpha} G(\ell_i)  \\
&\hspace{0.4cm} \times 2 \sum_{{L \atop l_1+l_2+L \text{ even}} \atop  l_3+l_4+L \text{ even}} (2L+1) \left( \begin{matrix}  \ell_1 & \ell_2 &L \\   0 & 0 &0  \end{matrix}  \right)^2 \left( \begin{matrix}  \ell_3 & \ell_4 &L \\   0 & 0 &0  \end{matrix}  \right)^2.
\end{align*}
Since $\max \{|\ell_1-\ell_2|, |\ell_3-\ell_4|\} \le L \le \min \{ \ell_1+\ell_2, \ell_3+\ell_4 \}$ where $ |\ell_i-\ell_k| \ge 0$ and $\ell_i+\ell_k \le 4 \ell$, we can write
\begin{align*}
Var[\nu_{j;4}]&=4! 8 \pi^2 \ell^{-4 \alpha+10}  B_\ell^{-4}   \sum_{\ell_1 \ell_2 \ell_3 \ell_4} \int_{\frac{\ell_1}{\ell}}^{\frac{\ell_1+1}{\ell}} \cdots \int_{\frac{\ell_4}{\ell}}^{\frac{\ell_4+1}{\ell}}  \prod_{i=1}^4   b^2\left(\frac{\lfloor \ell x_i\rfloor }{\ell} \right)    \frac{2 \lfloor \ell x_i\rfloor  +1}{2 \ell}\\   
&\hspace{0.4cm} \times \left( \frac{ \lfloor \ell x_i\rfloor }{\ell} \right)^{-\alpha} \frac{G(\lfloor \ell x_i\rfloor )}{2 \pi} \; 2 \sum_{L=0  \atop {l_1+l_2+L \text{ even} \atop  l_3+l_4+L \text{ even}}}^{4 \ell} \int_{\frac  L \ell}^{\frac {L+1} \ell} \frac{2 \lfloor \ell y \rfloor +1}{2 \ell} \\
 &\hspace{0.4cm} \times  \left( \begin{matrix}  \lfloor \ell x_1\rfloor  & \lfloor \ell x_2\rfloor  &\lfloor \ell y\rfloor  \\   0 & 0 &0  \end{matrix}  \right)^2 \left( \begin{matrix}  \lfloor \ell x_3\rfloor  & \lfloor \ell x_4\rfloor  &\lfloor \ell y\rfloor  \\   0 & 0 &0  \end{matrix}  \right)^2 dy d x_1\cdots d x_4. 
\end{align*}
Then, by dominated convergence and Lemma \ref{wigner}, we have  
\begin{align*}
& \hspace{3cm}\lim_{\ell \to \infty} \ell^2 Var[\nu_{j;4}]\\
&=\lim_{\ell \to \infty}  4! 8 \pi^2 \ell^{-4 \alpha+8}  B_\ell^{-4}   \sum_{\ell_1 \ell_2 \ell_3 \ell_4} \int_{\frac{\ell_1}{\ell}}^{\frac{\ell_1+1}{\ell}} \cdots \int_{\frac{\ell_4}{\ell}}^{\frac{\ell_4+1}{\ell}}  \prod_{i=1}^4   b^2\left(\frac{\lfloor \ell x_i\rfloor }{\ell} \right)    \frac{2 \lfloor \ell x_i\rfloor  +1}{2 \ell} \\  
&\hspace{0.4cm} \times \left( \frac{ \lfloor \ell x_i\rfloor }{\ell} \right)^{-\alpha} \frac{G(\lfloor \ell x_i\rfloor )}{2 \pi}  \; 2 \sum_{L=0  \atop {l_1+l_2+L \text{ even} \atop  l_3+l_4+L \text{ even}}}^{4 \ell} \int_{\frac  L \ell}^{\frac {L+1} \ell} \frac{2 \lfloor \ell y \rfloor +1}{2 \ell}   \frac{4 e^2}{\pi^2} \\
&\hspace{0.4cm} \times \frac{\ell^2}{\sqrt{(-\lfloor \ell x_1\rfloor+\lfloor \ell x_2\rfloor+\lfloor \ell y \rfloor)(\lfloor \ell x_1\rfloor-\lfloor \ell x_2\rfloor+\lfloor \ell y \rfloor) (\lfloor \ell x_1\rfloor+\lfloor \ell x_2\rfloor-\lfloor \ell y \rfloor)}} \\
&\hspace{0.4cm} \times  \frac{\lfloor \ell x_1\rfloor+\lfloor \ell x_2\rfloor+\lfloor \ell y \rfloor}{(\lfloor \ell x_1\rfloor+\lfloor \ell x_2\rfloor+\lfloor \ell y \rfloor+1)^{\frac 3 2}}  \left( 1+\frac{1}{\lfloor \ell x_1\rfloor+\lfloor \ell x_2\rfloor+\lfloor \ell y \rfloor}  \right)^{-(\lfloor \ell x_1\rfloor+\lfloor \ell x_2\rfloor+\lfloor \ell y \rfloor)}\\
&\hspace{0.4cm} \times \frac{\ell^2}{\sqrt{(-\lfloor \ell x_3\rfloor+\lfloor \ell x_4\rfloor+\lfloor \ell y \rfloor)(\lfloor \ell x_3\rfloor-\lfloor \ell x_4\rfloor+\lfloor \ell y \rfloor) (\lfloor \ell x_3\rfloor+\lfloor \ell x_4\rfloor-\lfloor \ell y \rfloor)}} \\
&\hspace{0.4cm} \times  \frac{\lfloor \ell x_3\rfloor+\lfloor \ell x_4\rfloor+\lfloor \ell y \rfloor}{(\lfloor \ell x_3\rfloor+\lfloor \ell x_4\rfloor+\lfloor \ell y \rfloor+1)^{\frac 3 2}}  \left( 1+\frac{1}{\lfloor \ell x_3\rfloor+\lfloor \ell x_4\rfloor+\lfloor \ell y \rfloor}  \right)^{-(\lfloor \ell x_3\rfloor+\lfloor \ell x_4\rfloor+\lfloor \ell y \rfloor)}\\
&\hspace{0.4cm} \times \ind_{P_3}(x_1,x_2,y)  \ind_{P_3}(x_3,x_4,y)  dy d x_1\cdots d x_4\\
&= \lim_{\ell \to \infty} 4! 8 \pi^2 \ell^{-4 \alpha+8}  B_\ell^{-4}   \int_{\frac 1 2}^{\frac{2 \ell+1}{\ell}} \cdots \int_{\frac 1 2}^{\frac{2 \ell+1}{\ell}} \prod_{i=1}^4   b^2\left(\frac{\lfloor \ell x_i\rfloor }{\ell} \right)    \frac{2 \lfloor \ell x_i\rfloor  +1}{2 \ell} \\
&\hspace{0.4cm} \times  \left( \frac{ \lfloor \ell x_i\rfloor }{\ell} \right)^{-\alpha} \frac{G(\lfloor \ell x_i\rfloor )}{2 \pi}  \int_{0}^{\frac {4 \ell+1} \ell} \frac{2 \lfloor \ell y \rfloor +1}{2 \ell}   \frac{4 e^2}{\pi^2} \\
&\hspace{0.4cm} \times \frac{\ell^2}{\sqrt{(-\lfloor \ell x_1\rfloor+\lfloor \ell x_2\rfloor+\lfloor \ell y \rfloor)(\lfloor \ell x_1\rfloor-\lfloor \ell x_2\rfloor+\lfloor \ell y \rfloor) (\lfloor \ell x_1\rfloor+\lfloor \ell x_2\rfloor-\lfloor \ell y \rfloor)}} \\
&\hspace{0.4cm} \times  \frac{\lfloor \ell x_1\rfloor+\lfloor \ell x_2\rfloor+\lfloor \ell y \rfloor}{(\lfloor \ell x_1\rfloor+\lfloor \ell x_2\rfloor+\lfloor \ell y \rfloor+1)^{\frac 3 2}}  \left( 1+\frac{1}{\lfloor \ell x_1\rfloor+\lfloor \ell x_2\rfloor+\lfloor \ell y \rfloor}  \right)^{-(\lfloor \ell x_1\rfloor+\lfloor \ell x_2\rfloor+\lfloor \ell y \rfloor)}\\
&\hspace{0.4cm} \times \frac{\ell^2}{\sqrt{(-\lfloor \ell x_3\rfloor+\lfloor \ell x_4\rfloor+\lfloor \ell y \rfloor)(\lfloor \ell x_3\rfloor-\lfloor \ell x_4\rfloor+\lfloor \ell y \rfloor) (\lfloor \ell x_3\rfloor+\lfloor \ell x_4\rfloor-\lfloor \ell y \rfloor)}} \\
&\hspace{0.4cm} \times  \frac{\lfloor \ell x_3\rfloor+\lfloor \ell x_4\rfloor+\lfloor \ell y \rfloor}{(\lfloor \ell x_3\rfloor+\lfloor \ell x_4\rfloor+\lfloor \ell y \rfloor+1)^{\frac 3 2}}  \left( 1+\frac{1}{\lfloor \ell x_3\rfloor+\lfloor \ell x_4\rfloor+\lfloor \ell y \rfloor}  \right)^{-(\lfloor \ell x_3\rfloor+\lfloor \ell x_4\rfloor+\lfloor \ell y \rfloor)}\\
&\hspace{0.4cm} \times \ind_{P_3}(x_1,x_2,y)  \ind_{P_3}(x_3,x_4,y)  dy d x_1\cdots d x_4.
\end{align*}
Once again, applying dominated convergence and Lemma \ref{4},  we have the statement. \end{proof}

\section{Quantitative Central Limit Theorems for $\nu_{j;q}$} \label{TV}
We start by recalling that $H_q(\tilde{\beta}_j(x))$ belongs to
the $q$-th order Wiener chaos and so does the linear transform
$\nu_{j;q}$. Inside a fixed Wiener chaos it is possible to get
explicit estimates on the speed of convergence to the Gaussian law
for the Kolmogorov, Total Variation and Wasserstein distance by
applying Proposition \ref{N&P} and by explicitly relate norms of
Malliavin operators with moments and cumulants. In fact, for $\cal{N}$
standard Gaussian, we have
$$d \left(  \frac{\nu_{j;q}}{\sqrt{Var(\nu_{j;q})}}, {\cal N} \right) \le 2 \sqrt{\frac{q-1}{3 q} \frac{{cum}_4(\nu_{j;q})}{Var^2(\nu_{j;q})}},$$
where $d$ is the Kolmogorov, Total Variation or Wasserstein
distance and ${cum}_4$ is the forth-order cumulant of
$\nu_{j;q}$. See \cite{noupebook}, Theorem 5.2.6 for more
discusion and a full proof.

Quantitative Central Limit Theorems for $\nu_{j;q}$ then follow
easily from the results of Section \ref{variance} and by computing
the forth-order cumulant as in \cite{marinuccivadlamani} Section
5.1. The arguments are indeed quite standard but nevertheless 
 for completeness we report them below.

We start by expressing the $4$-th order cumulant as an integral over
$(S^2)^4$, using the well-known Diagram formula, see
\cite{marpecbook}, Proposition 4.15 for further details.

 Fix a set of integers $\alpha_1, \dots, \alpha_p$, a diagram is a graph with $\alpha_1$ vertexes labelled by $1$, $\alpha_2$ vertexes  labelled by $2$, $\dots$  $\alpha_p$ vertexes labelled by $p$, such that each vertex has degree $1$. We can view the vertexes as belonging to $p$ different rows and the edges may connect only vertexes with different labels, i.e. there are no flat edges on the same row. The set of such graphs that are connected (i.e. such that it is not possible to partition the vertexes into two subsets $A$ and $B$ such that no edge connect a vertex in $A$ with a vertex in $B$) is denoted by $\Gamma_c(\alpha_1,\dots, \alpha_p)$. Given a diagram $\gamma \in \Gamma_c$, $\eta_{ik}(\gamma)$  is the number of edges between the vertexes labelled by $i$ and the vertexes labelled by $k$ in $\gamma$. The following proposition holds:
\begin{proposition}[Diagram formula for Hermite polynomials] \label{cum}
Let $(Z_1,\dots, Z_p)$ be e centered Gaussian vector, and let $H_{l_1}, \dots,H_{l_p}$ be Hermite polynomials of degrees $l_1,\dots,l_p (\ge 1)$ respectively. Then
$${cum}(H_{l_1}(Z_1),\dots,H_{l_p}(Z_p))=\sum_{\gamma \in \Gamma_c(l_1,\dots,l_p)} \prod_{1 \le i \le j \le p} \{\mathbb{E}[Z_i Z_j]\}^{\eta_{ij}(\gamma)}.$$
\end{proposition}
\noindent for a proof see \cite{pectaq}, Section 7.3. 
\begin{theorem} \label{ccc}
For ${\cal N}$ standard Gaussian variable and for all $q$ such that $c_q>0$, as $j \to \infty$, we have that
\begin{align*}
d_{TV} \left(  \frac{\nu_{j;q}}{\sqrt{Var(\nu_{j,q})}}, {\cal N} \right),d_{W} \left(  \frac{\nu_{j;q}}{\sqrt{Var(\nu_{j,q})}}, {\cal N} \right) =
O(2^{-2j}).
\end{align*}
\end{theorem}
\begin{proof} In view of
Proposition \ref{cum}, for $p=4$ and $l_1=\cdots=l_4=q$, we obtain
\begin{align*}
 {cum}_4[\nu_{j;q}]&={cum}_4 \left[ \int_{S^2} H_q(\tilde{\beta}_j(x_1)) d \sigma(x_1) \cdots \int_{S^2} H_q(\tilde{\beta}_j(x_4)) d \sigma(x_4) \right]\\
&=\int_{(S^2)^4} {cum}_4[H_q(\tilde{\beta}_j(x_1)) \cdots H_q(\tilde{\beta}_j(x_4)) ] \; d\sigma(x_1) \cdots d\sigma(x_4)  \\
&=\int_{(S^2)^4} \sum_{\gamma \in \Gamma_c(q,q,q,q)} \prod_{(i,k) \in \gamma}\{\mathbb{E}[\tilde{\beta}_j(x_i)  \tilde{\beta}_j(x_k) ]\}^{\eta_{ik}(\gamma)} \; d \sigma(x_1) \cdots d \sigma(x_4)  \\
&=\frac{1}{B_j^{2q}}\int_{(S^2)^4} \sum_{\gamma \in \Gamma_c(q,q,q,q)} \prod_{(i,k) \in \gamma}\{\mathbb{E}[{\beta}_j(x_i)  {\beta}_j(x_k) ]\}^{\eta_{ik}(\gamma)} \; d \sigma(x_1) \cdots d \sigma(x_4),
\end{align*}
since $\sum_{(i,k) \in \gamma} \eta_{ik}(\gamma)=2 q$.  Now we apply formula (\ref{mu1}) and we obtain
\begin{align*}
& \hspace{3cm} {cum}_4[ \nu_{j;q}]\\
&\le \frac{1}{B_j^{2q}}  \sum_{\gamma \in \Gamma_c(q,q,q,q)} \int_{(S^2)^4} \prod_{(i,k) \in \gamma} \left\{ \frac{C_M}{(1+2^j d( x_i , x_k  ))^M}   \sum_{l} b^2(\frac{l}{2^j})  C_l  \frac{2l+1}{4 \pi} \right\}^{\eta_{ik}(\gamma) } \; d \sigma(x_1) \cdots d \sigma(x_4)  \\
&=  \tilde{C}_M ^{2 q}  \sum_{\gamma \in \Gamma_c(q,q,q,q)} \int_{(S^2)^4} \prod_{(i,k) \in \gamma}   \frac{1}{(1+2^j d( x_i , x_k  ))^{M {\eta_{ik}(\gamma)}} }   \; d \sigma(x_1) \cdots d \sigma(x_4).
\end{align*}
To compute the integral we note that for spherical symmetry we can assume without loss of generality that e.g. $x_3$ is the North Pole denoted by $p_N$, and we get
\begin{align*}
&\int_{(S^2)^4} \prod_{(i,k) \in \gamma}   \frac{1}{(1+2^j d( x_i , x_k  ))^{M {\eta_{ik}(\gamma)}} }   \; d\sigma(x_1) \cdots d \sigma(x_4)\\
&\le  \int_{(S^2)^4}   \frac{1}{(1+2^j d( x_1 , x_2 ))^M }     \frac{1}{(1+2^j d( x_2 , x_3))^M }   \frac{1}{(1+2^j d(x_3 , x_4 ))^M}   \frac{1}{(1+2^j d(x_1 , x_4 ))^M}   \; d \sigma(x_1) \cdots d \sigma(x_4)  \\
&\le 4 \pi  \int_{(S^2)^3}   \frac{1}{(1+2^j d( x_1 , x_2 ))^M }     \frac{1}{(1+2^j d( x_2 , p_N))^M }   \frac{1}{(1+2^j d(p_N , x_4 ))^M}     \; d\sigma(x_1) d \sigma(x_2)  d \sigma(x_4) \\
&\le 4 \pi C 2^{-2j} \int_{(S^2)^2}   \frac{1}{(1+2^j d( x_1 , x_2 ))^M }     \frac{1}{(1+2^j d( x_2 , p_N))^M }        \; d \sigma(x_1) d \sigma(x_2)   \\
&\le \text{const }2^{-6j}
\end{align*}
since, for example, for $M>2$
\begin{align*}
\int_{S^2}  \frac{1}{(1+2^j d(p_N , x_4 ))^M} d\sigma(x_4)&=\int_{0}^{2 \pi} d \phi \int_{0}^\pi \frac{\theta \sin \theta}{(1+2^j  \theta)^M} d \theta \le 2 \pi  \int_{0}^\infty \frac{\theta}{(1+2^j  \theta)^M} d \theta\\
&=2 \pi \left[\int_{0}^{2^{-j}} \frac{\theta}{(1+2^j  \theta)^M} d \theta+\int_{2^{-j}}^\infty \frac{\theta}{(1+2^j  \theta)^M} d \theta  \right]\\
&\le 2 \pi \left[\int_{0}^{2^{-j}} \theta d \theta+ 2^{-j M} \int_{2^{-j}}^\infty \theta^{1-M}  d \theta  \right]\\
&=2 \pi \left[ 2^{-1-2j}+ \frac{2^{-2j}}{M-2} \right] \le \text{const } 2^{-2j}.
\end{align*}
\end{proof}

\section{A Quantitative Central Limit Theorem for the empirical measure}

In the next  theorem we obtain a bound on the Wasserstein distance for the speed of convergence of ${{\Phi}_j(z)}$ to the Gaussian law.

\begin{theorem} \label{WD}
For $\cal N$ standard Gaussian, as  $j \to \infty$ we have
$$d_W \left( \frac{{\Phi}_j(z)}{\sqrt{\text{Var} [{\Phi}_j(z)]}}, \cal{N}  \right)=O\left(\frac{1}{\sqrt[4]{j}} \right).$$
\end{theorem}
\noindent We start by proving the following lemma.
\begin{lemma} \label{lemma-lemma}
For integers $q, q' \ge 2$ we have that
\begin{align*}
\mathbb{E}[(\langle D  \nu_{j;q},-DL^{-1} \nu_{j;q'}  \rangle_{\mathfrak H} )^2]&\le \text{const } 2^{-6j}  q^2 \sum_{r=1}^{q \wedge q' } (r-1)!^2 \binom{q-1}{r-1}^2 \binom{q'-1}{r-1}^2 (q+q'-2r)!, \\
\text{Var}[\langle D  \nu_{j;q},-DL^{-1}  \nu_{j;q}
\rangle_{\mathfrak H} ]&\le \text{const } 2^{-6j} q^2
\sum_{r=1}^{q-1 } (r-1)!^2 \binom{q-1}{r-1}^4  (2q-2r)!.
\end{align*}
\end{lemma}
\begin{proof}
Since $H_q(\tilde{\beta}_j(x))$ is in the $q$-th order Wiener chaos, from (\ref{HI}), we obtain
\begin{align*}
 \nu_{j;q}&=\int_{S^2} d x \int_{(S^2)^q} \prod_{i=1}^q {\tilde \Theta}_j (\langle x,y_i \rangle) W(d \sigma(y_i)) =  \int_{(S^2)^q} g_{q,j} (y_1,\dots,y_{q}) W(d \sigma(y_1)) \cdots W(d \sigma(y_q))\\
&=I_q( g_{q,j} (y_1,\dots,y_{q})),
\end{align*}
where
\begin{align*}
 g_{q,j} (y_1,\dots,y_{q}) = \int_{S^2}  \prod_{i=1}^q {\tilde \Theta}_j (\langle x,y_i \rangle)  d \sigma(x)
\end{align*}
and, from formula (\ref{DI}), 
\begin{align*}
D  \nu_{j;q}=\frac{q!}{(q-1)!} I_{q-1} ( g_{q,j}
(y_1,\dots,y_{q-1},z))=q \int_{(S^2)^{q-1}}
g_{q,j}(y_1,\dots,y_{q-1},z) W(d \sigma(y_1)) \dots W(d \sigma(y_{q-1})).
\end{align*}
Applying the definition  of the pseudo-inverse of $L$, we obtain
\begin{align*}
\langle D  \nu_{j;q},-DL^{-1} \nu_{j;q'})  \rangle_{\mathfrak H}&= \frac{1}{q'} \langle D  \nu_{j;q},D \nu_{j;q'}  \rangle_{\mathfrak H}\\
&=q \langle I_{q-1}(g_{q,j}(y_1,\dots,y_{q-1},z)), I_{q'-1}(g_{q',j}(y_1,\dots,y_{q'-1},z)) \rangle_{\mathfrak H}\\
&=q \int_{S^2}  I_{q-1} ( g_{q,j} (y_1,\dots,y_{q-1},z)) I_{q'-1} ( g_{q',j} (y_1,\dots,y_{q'-1},z)) d\sigma(z)
\end{align*}
and by the multiplication formula (\ref{P}) 
\begin{align*}
& \hspace{3cm} \langle D  \nu_{j;q},-DL^{-1} \nu_{j;q'}  \rangle_{\mathfrak H}\\
&=q \sum_{r=0}^{q \wedge q'-1 } r! \binom{q-1}{r} \binom{q'-1}{r} \int_{S^2} I_{q+q'-2-2 r}(g_{q,j} (y_1,\dots,y_{q-1},z) {\tilde{\otimes}_r} g_{q',j} (y_1,\dots,y_{q'-1},z) ) d\sigma(z)\\
&=q \sum_{r=0}^{q \wedge q'-1 } r! \binom{q-1}{r} \binom{q'-1}{r} I_{q+q'-2-2 r}(g_{q,j} (y_1,\dots,y_q) {\tilde{\otimes}_{r+1}} g_{q',j} (y_1,\dots,y_{q'})) \\
&=q \sum_{r=1}^{q \wedge q' } (r-1)! \binom{q-1}{r-1} \binom{q'-1}{r-1} I_{q+q'-2 r}(g_{q,j} (y_1,\dots,y_q) {\tilde{\otimes}_{r}}g_{q',j} (y_1,\dots,y_{q'}) ).
\end{align*}
From the isometry property (\ref{EI}) we have
\begin{align*}
& \hspace{3cm} \mathbb{E}[(\langle D  \nu_{j;q},-DL^{-1} \nu_{j;q'}  \rangle_{\mathfrak H} )^2]\\
&=q^2 \sum_{r=1}^{q \wedge q' } (r-1)!^2 \binom{q-1}{r-1}^2 \binom{q'-1}{r-1}^2 (q+q'-2r)! ||g_{q,j} (y_1,\dots,y_q) {\tilde{\otimes}_{r}}g_{q',j} (y_1,\dots,y_{q'}) ||^2_{{\mathfrak H}^{\otimes q+q'-2r}}\\
&\le q^2 \sum_{r=1}^{q \wedge q' } (r-1)!^2 \binom{q-1}{r-1}^2 \binom{q'-1}{r-1}^2 (q+q'-2r)! ||g_{q,j} (y_1,\dots,y_q) {{\otimes}_{r}}g_{q',j} (y_1,\dots,y_{q'}) ||^2_{{\mathfrak H}^{\otimes q+q'-2r}}.
\end{align*}
Applying Lemma 6.2.1 in \cite{noupebook}, we write
\begin{align*}
& \hspace{3cm} \text{Var}[\langle D  \nu_{j;q},-DL^{-1} \nu_{j,q}  \rangle_{\mathfrak H} ]\\
&\le q^2 \sum_{r=1}^{q-1 } (r-1)!^2 \binom{q-1}{r-1}^4  (2q-2r)!
||g_{q,j} (y_1,\dots,y_q) {{\otimes}_{q-r}}g_{q,j}
(y_1,\dots,y_{q}) ||^2_{{\mathfrak H}^{\otimes 2r}}.
\end{align*}
We determine now the explicit form for the contractions:
\begin{align*}
& \hspace{3cm} g_{q,j} (y_1,\dots,y_q) {{\otimes}_{r}}g_{q',j} (y_1,\dots,y_{q'})\\
&=\int_{(S^2)^r} g_{q,j} (y_1,\dots,y_{q-r},t_1,\dots,t_r) g_{q',j} (y_{q-r+1},\dots,y_{q+q'-2r},t_1,\dots,t_r) d \sigma(t_1) \dots d \sigma(t_r) \\
&= \int_{(S^2)^r}  [\int_{S^2} \prod_{n=1}^{q-r} {\tilde \Theta}_j(\langle x_1,y_n \rangle)   \prod_{i=1}^{r} {\tilde \Theta}_j(\langle x_1,t_i \rangle)  d \sigma(x_1)] \\
& \;\; \times [\int_{S^2} \prod_{m=q-r+1}^{q+q'-2r} {\tilde \Theta}_j(\langle x_2,y_m \rangle)   \prod_{i=1}^{r} {\tilde \Theta}_j(\langle x_2,t_i \rangle)  d \sigma(x_2)] d \sigma(t_1) \dots d \sigma(t_r) \\
&=B_j^{-\frac{q+q'}{2}}  \int_{(S^2)^2} d \sigma(x_1) d \sigma(x_2) \prod_{n=1}^{q-r} { \Theta}_j(\langle x_1,y_n \rangle)  \prod_{m=q-r+1}^{q+q'-2r} { \Theta}_j(\langle x_2,y_m \rangle)  \\
& \;\; \times \int_{(S^2)^r}  \prod_{i=1}^{r} { \Theta}_j(\langle x_1,t_i \rangle)  { \Theta}_j(\langle x_2,t_i \rangle) d\sigma(t_1) \dots d \sigma(t_r) \\
&=B_j^{-\frac{q+q'}{2}} \int_{(S^2)^2} d\sigma(x_1) d \sigma(x_2) \prod_{n=1}^{q-r} { \Theta}_j(\langle x_1,y_n \rangle)  \prod_{m=q-r+1}^{q+q'-2r} { \Theta}_j(\langle x_2,y_m \rangle)   \rho^r_j(\langle x_1, x_2 \rangle),
\end{align*}
for $\Theta_j$ and $\rho_j$ as in (\ref{Psi}) and (\ref{Theta}). It follows that
\begin{align*}
&||g_{q,j} (y_1,\dots,y_q) {{\otimes}_{r}}g_{q',j} (y_1,\dots,y_{q'}) ||^2_{{\mathfrak H}^{\otimes q+q'-2r}}\\
&=B_j^{-(q+q')}\int_{(S^2)^{q+q'-2r}} d \sigma(y_1) \cdots  d \sigma(y_{q+q'-2r}) [\int_{(S^2)^4}   \prod_{n=1}^{q-r} { \Theta}_j(\langle x_1,y_n \rangle)  { \Theta}_j(\langle x_3,y_n \rangle) \\
&\;\;\; \times  \prod_{m=q-r+1}^{q+q'-2r} { \Theta}_j(\langle x_2,y_m \rangle)  { \Theta}_j(\langle x_4,y_m \rangle) 
  \rho^r_j(\langle x_1, x_2 \rangle)   \rho^r_j(\langle x_3, x_4 \rangle) d\sigma(x_1) \cdots  d\sigma(x_4)]\\
&= B_j^{-(q+q')} \int_{(S^2)^4}   \rho^{q-r}_j(\langle x_1, x_3 \rangle)   \rho^{q'-r}_j(\langle x_2, x_4 \rangle)   \rho^r_j(\langle x_1, x_2 \rangle)   \rho^r_j(\langle x_3, x_4 \rangle)   d\sigma(x_1) \cdots d\sigma(x_4).
\end{align*}
Since $\rho_j(\langle x, y \rangle)\le B_j$ and from (\ref{mu1})
$$\int_{S^2} \rho^p_j(\langle x, y \rangle) d \sigma(x) \le \int_{S^2} \left(  \frac{C_M}{(1+2^j d(x,y))^M} B_j \right)^p d \sigma(x)  \le B_j^p C_M^p 2^{-2j},$$
we have
\begin{align*}
&||g_{q,j} (y_1,\dots,y_q) {{\otimes}_{r}}g_{q',j} (y_1,\dots,y_{q'}) ||^2_{{\mathfrak H}^{\otimes q+q'-2r}}\\
&= B_j^{-(q+q')} B_j^r \int_{(S^2)^4}   \rho^{q-r}_j(\langle x_1, x_3 \rangle)   \rho^{q'-r}_j(\langle x_2, x_4 \rangle)   \rho^r_j(\langle x_1, x_2 \rangle)     d\sigma(x_1) \cdots d \sigma(x_4)\\
& \le \text{const } 2^{-6j},
\end{align*}
and analogously
\begin{align*}
&||g_{q,j} (y_1,\dots,y_q) {{\otimes}_{q-r}}g_{q,j} (y_1,\dots,y_{q}) ||^2_{{\mathfrak H}^{\otimes 2r}}\\
&= B_j^{-2q} \int_{(S^2)^4}   \rho^{q-r}_j(\langle x_1, x_3 \rangle)   \rho^{q-r}_j(\langle x_2, x_4 \rangle)   \rho^r_j(\langle x_1, x_2 \rangle)   \rho^r_j(\langle x_3, x_4 \rangle)   d\sigma(x_1) \cdots d \sigma(x_4) \\ &  \le \text{const } 2^{-6j}.
\end{align*}
\end{proof}

\begin{proof} [{\it Proof of Theorem \ref{WD}}]
Let us introduce the following notation:
\begin{align*}
\tilde{\Phi}_{j,N}(z)&=2^{j} \sum_{q=2}^N \frac{{\cal J}_q(z)}{q!}
 \nu_{j;q} , \hspace{1cm}\sigma_N^2=\frac{\text{Var} [{\tilde
\Phi}_{j,N}(z)]}{\text{Var} [{2^j \Phi}_{j}(z)]}, \hspace{1cm}
{\cal N}_{N} \sim {\cal N} \left(0, \sigma_N^2 \right).
\end{align*}
We have that
\begin{align*}
d_W \left( \frac{{ \Phi}_j(z)}{\sqrt{\text{Var} [{ \Phi}_j(z)]}}, {\cal N}  \right)&\le d_W \left( \frac{{ \Phi}_j(z)}{\sqrt{\text{Var} [{ \Phi}_j(z)]}}, \frac{{\tilde \Phi}_{j,N}(z)}{\sqrt{\text{Var} [{2^j \Phi}_j(z)]}}  \right) \\
&\;\;+  d_W \left( \frac{{\tilde \Phi}_{j,N}(z)}{\sqrt{\text{Var} [{2^j \Phi}_j(z)]}} ,{\cal N}_N \right)+d_W \left({\cal N}_N ,{\cal N} \right).
\end{align*}
\begin{itemize}
\item {For the first term we apply the properties of the Wasserstein distance to get:}
\end{itemize}
\begin{align*}
&d_W \left( \frac{{ \Phi}_j(z)}{\sqrt{\text{Var} [{ \Phi}_j(z)]}}, \frac{{\tilde \Phi}_{j,N}(z)}{\sqrt{\text{Var} [{2^j \Phi}_j(z)]}}  \right)\\
&\;\;\; \le  \left\{  \mathbb{E}\left[  \frac{{ \Phi}_j(z)}{\sqrt{\text{Var} [{\Phi}_j(z)]}}- \frac{{\tilde \Phi}_{j,N}(z)}{\sqrt{\text{Var} [{2^j \Phi}_j(z)]}}   \right]^2 \right\}^{1/2}\\
&\;\;\; =\frac{1}{\sqrt{\text{Var} [{2^j \Phi}_j(z)]}}  \left\{ \mathbb{E} \left[  2^j \int_{S^2} \sum_{q=N+1}^\infty  \frac{{\cal J}_q(z)}{q!} H_q(\tilde{\beta}_j(x))  d \sigma(x) \right]^2  \right\}^{1/ 2},
\end{align*}
and since ${ 2^j \Phi}_j(z)-{\tilde \Phi}_{j,N}(z)$
belongs to the Hilbert space of Gaussian subordinated random
variables, with continuous inner product $\langle X,Y \rangle:= \mathbb{E}[X
Y]$, we have
\begin{align*}
d_W \left( \frac{{ \Phi}_j(z)}{\sqrt{\text{Var} [{
\Phi}_j(z)]}}, \frac{{\tilde \Phi}_{j,N}(z)}{\sqrt{\text{Var}
[{2^j \Phi}_j(z)]}}  \right) &  \le
\frac{1}{\sqrt{\text{Var}[{2^j \Phi}_j(z)]}}  \left\{
\sum_{q=N+1}^\infty \frac{{\cal J}_q^2(z)}{(q!)^2} 2^{2j}
\mathbb{E}[ \nu_{j;q}^2]  \right\}^{1/2}.
\end{align*}
Since for any finite $z$, as $q \to \infty$, the asymptotic formula  $e^{-\frac{z^2}{4}} H_q(z) \le \text{const } q^{\frac q 2} e^{- \frac{q}{2}}$ holds (see e.g. \cite{lebedev} formula (4.14.9)), by applying the Stirling's approximation to the factorial $(q-1)!$ we have (see \cite{V-HP}), 
$$\frac{{\cal J}_q^2(z)}{q!}=\frac{\phi(z) }{q!} [e^{-\frac {z^2} 4} H_{q-1}(z)]^2  \le \text{const }  \frac{\phi(z) }{q \sqrt{ q-1}}.$$
From this we obtain the first bound, in fact form Theorem \ref{00}, we have
\begin{align*}
d_W \left( \frac{{ \Phi}_j(z)}{\sqrt{\text{Var} [{ \Phi}_j(z)]}}, \frac{{\tilde \Phi}_{j,N}(z)}{\sqrt{\text{Var} [{2^j \Phi}_j(z)]}}  \right) &\le \text{const }  \left\{  \sum_{q=N+1}^\infty \frac{1}{q \sqrt{ q-1}} \frac{2^{2j}}{q!}  \mathbb{E}[ \nu_{j;q}^2] \right\}^{1/2}\\
&\le \text{const } N^{-
\frac 1 4} .
\end{align*}
\begin{itemize}
\item {To bound the second term, we apply now Proposition \ref{N&P} and we get}
\end{itemize}
\begin{align*}
& \hspace{3cm}d_W \left( \frac{{\tilde \Phi}_{j,N}(z)}{\sqrt{\text{Var} [{2^j \Phi}_j(z)]}} ,{\cal N}_N\right) \\
&\;\;\; \le \frac{\sqrt{2}}{\sigma_N \sqrt{\pi}} \frac{1}{\text{Var}[{2^j \Phi}_{j}(z)] } \mathbb{E} [|\text{Var}[{\tilde \Phi}_{j,N}(z)] - \langle D {\tilde \Phi}_{j,N}(z),-DL^{-1} {\tilde \Phi}_{j,N}(z)  \rangle_{\mathfrak H} |].
\end{align*}
Since, in view of (\ref{EH}), we have
\begin{align*}
\text{Var}[{\tilde \Phi}_{j,N}(z)]& =\sum_{q=2}^N \sum_{q'=2}^N \frac{{\cal J}_q(z)}{q!} \frac{{\cal J}_{q'}}{q'!} 2^{2j} \text{Cov} [ \nu_{j;q}, \nu_{j;q'}]\\
&=\sum_{q=2}^N \sum_{q'=2}^N \frac{{\cal J}_q(z)}{q!} \frac{{\cal J}_{q'}}{q'!} 2^{2j} \delta_{q}^{q'} q! \int_{S^2 \times S^2} \{ \mathbb{E}[{\tilde \beta}_j(x) {\tilde \beta}_j(y)] \}^q d \sigma(x) d \sigma(y)\\
&=\sum_{q=2}^N \frac{{\cal J}_q^2(z)}{(q!)^2} 2^{2j} {\text
Var}[ \nu_{j;q}],
\end{align*}
we write
\begin{align*}
& \hspace{3cm}d_W \left( \frac{{\tilde \Phi}_{j,N}(z)}{\sqrt{\text{Var} [{2^j \Phi}_j(z)]}} ,{\cal N}_N\right)\\
 &\;\;\; \le \frac{\sqrt{2}}{\sigma_N \sqrt{\pi}} \frac{2^{2j}}{\text{Var}[{2^j \Phi}_{j}(z)] }\sum_{q=2}^N \frac{{\cal J}_q(z)}{q!}  \mathbb{E} [| \frac{{\cal J}_q(z)}{q!} {\text Var}[\nu_{j;q}]  -\sum_{q'=2}^N \frac{{\cal J}_{q'}(z)}{q'!}  \langle D  \nu_{j;q},-DL^{-1} \nu_{j;q'}  \rangle_{\mathfrak H} |] \\
&\;\;\; \le \frac{\sqrt{2}}{\sigma_N \sqrt{\pi}} \frac{2^{2j}}{\text{Var}[{2^j \Phi}_{j}(z)] }\sum_{q=2}^N \frac{{\cal J}^2_q(z)}{(q!)^2}  \mathbb{E} [|  {\text Var}[\nu_{j;q}]  - \langle D  \nu_{j;q},-DL^{-1}  \nu_{j;q}  \rangle_{\mathfrak H} |] \\
&\;\;\; \;\;\;+\frac{\sqrt{2}}{\sigma_N \sqrt{\pi}}
\frac{2^{2j}}{\text{Var}[{2^j \Phi}_{j}(z)] }\sum_{q=2}^N
\frac{{\cal J}_q(z)}{q!} \sum_{q \ne q'} \frac{{\cal J}_{q'}(z)}{{q'}!}
\mathbb{E} [|  \langle D  \nu_{j;q},-DL^{-1} \nu_{j;q'}
\rangle_{\mathfrak H} |].
\end{align*}
By Theorem 2.9.1 in \cite{noupebook} and by Cauchy-Schwartz inequality, we have
\begin{align*}
& \hspace{3cm}d_W \left( \frac{{\tilde \Phi}_{j,N}(z)}{\sqrt{\text{Var} [{2^j \Phi}_j(z)]}} ,{\cal N}_N \right)\\
 &\;\;\; \le \frac{\sqrt{2}}{\sigma_N \sqrt{\pi}} \frac{2^{2j}}{\text{Var}[{2^j \Phi}_{j}(z)] }\sum_{q=2}^N \frac{{\cal J}^2_q(z)}{(q!)^2} \{  {\text Var} [   \langle D  \nu_{j;q},-DL^{-1}  \nu_{j;q}  \rangle_{\mathfrak H} ] \}^{1/2} \\
&\;\;\; \;\;\;+\frac{\sqrt{2}}{\sigma_N \sqrt{\pi}}
\frac{2^{2j}}{\text{Var}[{2^j \Phi}_{j}(z)] }\sum_{q=2}^N
\frac{{\cal J}_q(z)}{q!} \sum_{q \ne q'} \frac{{\cal J}_{q'}(z)}{{q'}!}  \{
\mathbb{E} [(  \langle D  \nu_{j;q},-DL^{-1} \nu_{j;q'}
\rangle_{\mathfrak H} )^2] \}^{1/2}.
\end{align*}
Finally, in view of Lemma \ref{lemma-lemma}, we  write
\begin{align*}
&\hspace{3cm}d_W \left( \frac{{\tilde \Phi}_{j,N}(z)}{\sqrt{\text{Var} [{2^j \Phi}_j(z)]}} ,{\cal N}_N\right)\\
 &\le \text{const } \frac{\sqrt{2}}{\sigma_N \sqrt{\pi}} \frac{2^{2j}}{\text{Var}[{2^j \Phi}_{j}(z)] } 2^{-3j} \left\{ \sum_{q=2}^N \frac{{\cal J}^2_q(z)}{(q!)^2} q \sqrt{ \sum_{r=1}^{q-1 } (r-1)!^2 \binom{q-1}{r-1}^4  (2q-2r)! } \right.\\
&\left.\;\;\;+ \sum_{q=2}^N \frac{{\cal J}_q(z)}{q!} \sum_{q \ne q'} \frac{{\cal J}_{q'}(z)}{{q'}!} q \sqrt{\sum_{r=1}^{q \wedge q' } (r-1)!^2 \binom{q-1}{r-1}^2 \binom{q'-1}{r-1}^2 (q+q'-2r)! } \right\}.
\end{align*}
We now bound the two sums by  reproducing in our case calculations analog to those performed in \cite{V-HP}:
\begin{align} \label{q1}
&\sum_{r=1}^{q \wedge q' } (r-1)!^2 \binom{q-1}{r-1}^2 \binom{q'-1}{r-1}^2 (q+q'-2r)! \nonumber\\
&= (q-1)! (q'-1)! \sum_{r=1}^{q \wedge q' }  \binom{q-1}{r-1} \binom{q'-1}{r-1} \binom{q+q'-2r}{q-r} \nonumber\\
&\le (q-1)! (q'-1)! \sum_{r=1}^{q \wedge q' }  \binom{q-1}{r-1}  \binom{q'-1}{r-1} 2^{q+q'-2r} \nonumber\\
&= (q-1)! (q'-1)! 2^{q+q'-2} \sum_{r=0}^{q \wedge q'-1 }  \binom{q-1}{r}  \binom{q'-1}{r} 2^{-2r} \nonumber \\
&\le (q-1)! (q'-1)! 2^{q+q'-2} \left[ \sum_{r=0}^{q \wedge q'-1 }  \binom{q-1}{r} 2^{-r}\right] \left[  \sum_{r=0}^{q \wedge q'-1 } \binom{q'-1}{r} 2^{-r} \right] \nonumber \\
&\le (q-1)! (q'-1)! 2^{q+q'-2} \left[ \sum_{r=0}^{q-1 }  \binom{q-1}{r} 2^{-r}\right] \left[  \sum_{r=0}^{ q'-1 } \binom{q'-1}{r} 2^{-r} \right] \nonumber \\
&= (q-1)! (q'-1)! 2^{q+q'-2} (1+1/2)^{q-1} (1+1/2)^{q'-1}=  (q-1)! (q'-1)! 3^{q+q'-2},
\end{align}
and likewise
\begin{align} \label{q2}
&\sum_{r=1}^{q-1 } (r-1)!^2 \binom{q-1}{r-1}^4  (2q-2r)!\le [(q-1)!]^2 3^{2q-2}.
\end{align}
Since for any finite $z$, as $q \to \infty$, we have
$$\frac{{\cal J}_q(z)}{q!}=\frac{\phi(z) H_{q-1}(z) }{q!}  \le \text{const }  \frac{\sqrt{\phi(z)} }{\sqrt{q!} \sqrt{q}  (q-1)^{\frac 1 4}},$$
from (\ref{q1}) and (\ref{q2}), we obtain
\begin{align*}
&\sum_{q=2}^N \frac{{\cal J}_q(z)}{q!} \sum_{q \ne q'} \frac{{\cal J}_{q'}(z)}{{q'}!} q \sqrt{\sum_{r=1}^{q \wedge q' } (r-1)!^2 \binom{q-1}{r-1}^2 \binom{q'-1}{r-1}^2 (q+q'-2r)! } \\
&\;\;\; \le \sum_{q=2}^N \frac{{\cal J}_q(z)}{q!} \sum_{q \ne q'} \frac{{\cal J}_{q'}(z)}{{q'}!} q  \sqrt{  (q-1)! (q'-1)! 3^{q+q'-2}} \\
&\;\;\; \le \text{const } \sum_{q=2}^N \frac{3^{\frac {q-1} 2}}{\sqrt{q}} \sum_{q'=2}^N \frac{3^{\frac {q'-1} 2}}{\sqrt{q'}} \le \text{const } 3^N,
\end{align*}
and
\begin{align*}
& \sum_{q=2}^N \frac{{\cal J}^2_q(z)}{q!^2} q \sqrt{ \sum_{r=1}^{q-1 } (r-1)!^2 \binom{q-1}{r-1}^4  (2q-2r)! } \\& \;\;\; \le \sum_{q=2}^N \frac{{\cal J}^2_q(z)}{q!^2} q  (q-1)! 3^{q-1}\le \text{const } \sum_{q=2}^N \frac{3^{q-1}}{q} \le \text{const } 3^N.
\end{align*}
Since $\sigma_N^2 \ge 1- \text{const } N^{- \frac 1 2} $, it follows that the second term, as $N \to \infty$ is at most equal to
\begin{align*}
&d_W \left( \frac{{\tilde \Phi}_{j,N}(z)}{\sqrt{\text{Var} [{2^j \Phi}_j(z)]}} ,{\cal N}_N\right) \le \frac{\text{const }}{\sqrt{1-N^{- \frac 1 2}}} \frac{3^N}{2^{j}}.
\end{align*}
\begin{itemize}
\item For the third term, by Proposition 3.6.1 in
\cite{noupebook}, 
\end{itemize}
\begin{align*}
d_W \left({\cal N}_N, {\cal N}\right) &\le \sqrt{\frac{2}{\pi}} \frac{1}{1\vee
\sqrt{\frac{\text{Var} [{\tilde \Phi}_{j,N}(z)]} {\text{Var}
[{2^j \Phi}_{j}(z)]} }} \left|1-\frac{\text{Var} [{\tilde
\Phi}_{j,N}(z)]}{\text{Var} [{2^j \Phi}_{j}(z)]}
\right|=\sqrt{\frac{2}{\pi}} \left| \frac{\sum_{q=N+1}^\infty
\frac{{\cal J}_q^2(z)}{q!}  \frac{2^{2j} \mathbb{E}[\nu^2_{j,q}] }{q!}
}{\sum_{q=2}^\infty \frac{{\cal J}_q^2(z)}{q!}  \frac{2^{2j}
\mathbb{E}[\nu^2_{j,q}] }{q!} }  \right| \\
&\le \text{const }
N^{-\frac 1 2}.
\end{align*}
Summing up the three bounds, choosing the speed $N={\log(2^j )}/{2}$ and observing that the dominant term is $N^{-\frac 1 4}$,  we arrive at the statement.
\end{proof}

\begin{remark}
To obtain a bound on the Kolmogorov distance, it is enough to recall the standard inequality $$d_{Kol}(F,{\cal N}) \le 2 \sqrt{d_W(F,{\cal N})}.$$
\end{remark}

\end{document}